\documentclass{amsart} 

   \renewcommand{\footnote}[1]{
\textsuperscript{%ecriture en exposant
\addtocounter{footnote}{1}%incrementation du compteur
(\thefootnote)% impression au format "(compteur)"
}
\footnotetext{#1}% la note de bas de page
}

\usepackage{url}
\usepackage{bm}
\usepackage{amsmath}
\usepackage{amsthm}
\usepackage{trfsigns}
\usepackage{wasysym}
\usepackage{amssymb}
\usepackage{amsfonts}
\usepackage{graphicx}  
\usepackage[frenchb,english]{babel}
\usepackage[utf8]{inputenc}%a la place d'applemac utf8
\usepackage[T1]{fontenc}
\usepackage{mathrsfs}
\usepackage{pdfsync}
\usepackage{enumerate}
\usepackage{version}
\usepackage{calc}
\usepackage{subfigure}
 \usepackage{pstricks,pstricks-add,pst-math,pst-xkey}
 \usepackage{float,graphics}
 \usepackage{indentfirst}
 \usepackage{yhmath} %pour \wideparen
\theoremstyle{remark}
\newtheorem{defi}{Définition}

\newtheorem{remark}[defi]{Remark}%[chapter]
\newtheorem{notation}[defi]{Notation}

\theoremstyle{plain}% default%%%%%%%%%%%%%%%%%%%%%%%%%%
\newtheorem{thm}{Theorem}%[chapter]
\newtheorem{lem}[defi]{Lemma}

\newtheorem{prop}[defi]{Proposition}%[chapter]

\usepackage[dvips,colorlinks=true,linkcolor=blue,urlcolor=blue,citecolor=blue]{hyperref}

\def\v{\varepsilon}

\def\O{\Omega}

\def\p{\partial}

\def\o{\omega}
\def\R{\mathbb{R}}
\def\C{\mathscr{C}}
\def\S{\mathbb{S}}
\def\Div{{\rm  div}}

\def\n{\nabla}
\def\dist{{\rm dist}}
\def\tr{{\rm tr}}

\def\N{\mathbb{N}}

\def\I{\mathcal{I}}
\def\M{\mathcal{M}}

\def\1{\textrm{1\kern-0.25emI}}
\def\0{\textrm{0\kern-0.25emI}}

\def\H{\mathscr{H}}

\def\D{{\mathbb{D}}}

\def\K{\mathcal{K}}

\def\repereO{\mathcal{R}_0}
\def\repere{\mathcal{R}}

\def\Q{\mathcal{Q}}
\def\Pab{\mathcal{P}}
\def\Sab{\mathcal{S}}
\def\tos{\stackrel{\rm strictly}{\longrightarrow}}
\def\Proj{\Pi}
\def\({\left(}
\def\){\left(}
\def\supp{{\rm supp}}
\def\mint{\displaystyle\mathop{\,\,\rlap{---}\!\int}\nolimits}
\usepackage{yhmath}

\newlength\longarc
\newcommand\curve{\wideparen}%[2][0.3pt]{%
\def\vecc{\overrightarrow}
\title[Non existence of least gradient function for some characteristic functions in a $C^2$ domain]{Characteristic functions on the boundary of a planar domain need not be traces of least gradient functions}%
\author{Micka\"el { Dos\,Santos}}

\address{M. Dos Santos, 
              Universit\'e Paris Est-Cr\'eteil, 61 avenue du G\'en\'eral de Gaulle, 94010 Cr\'eteil Cedex \\
}
 \email{mickael.dos-santos@u-pec.fr} 
\begin{document}
\maketitle
\begin{abstract}
Given a smooth bounded planar domain $\O$, we construct a compact set on the boundary such that its characteristic function is not the trace of a least gradient function. This  generalizes the construction  of Spradlin and Tamasan \cite{ST1} when $\O$ is a disc.\end{abstract}
\tableofcontents
\section{Introduction}

We let $\O$ be a bounded $C^2$ domain of $\R^2$. For a function $h\in L^1(\p\O,\R)$, the least gradient problem with boundary datum $h$ consists in deciding whether
\begin{equation}\label{LeastGradH}
\inf\left\{\int_\O|D w|\;;\;w\in BV(\O)\text{ and }\tr_{\p\O}w=h\right\}
\end{equation}
is achieved or not.

In  the  above  minimization problem, $BV(\O)$ is the space of functions of bounded variation. It is the space of functions $w\in L^1(\O)$ having a distributional gradient $Dw$ which is a bounded Radon measure. %Namely, $BV(\O)=\{w\in L^1(\O,\R)\,;\,|Dw|(\O)<\infty\}$ with $|Dw|(\O)=\sup\{\int_\O w\Div(\varphi)\,;\,\varphi\in C^1_c(\O,\R^2),\,|\varphi|\leq1\}$.

If the {\it infimum} in \eqref{LeastGradH} is achieved, minimal functions are called {\it functions of least gradient}.

Sternberg, Williams and Ziemmer proved in \cite{SWZ1} that if $h:\p\O\to\R$ is a continuous map and if $\p\O$ satisfies a geometric properties then there exists a (unique) function of least gradient. For further use, we note that the geometric property is satisfied by Euclidean balls.

On the other hand, Spradlin and Tamasan \cite{ST1} proved that, for the disc $\O=\{x\in\R^2\,;\,|x|<1\}$, we may find a function $h_0\in L^1(\p\O)$ which is not continuous s.t. the {\it infimum} in \eqref{LeastGradH} is not achieved. The function $h_0$ is the characteristic function of a Cantor type set $\K\subset\S^1=\{x\in\R^2\,;\,|x|=1\}$

The goal of this article is to extend the main result of \cite{ST1} to a general $C^2$ bounded open set $\O\subset\R^2$.%give an example of map $h\in L^1(\p\O)$ s.t. the problem \eqref{LeastGradH} has no solution.

We prove the following theorem.

\begin{thm}\label{MainTHM}
Let $\O\subset\R^2$ be a bounded $C^2$ open set. Then there exists a measurable set $\K\subset\p\O$ such that the {\it infimum}
\begin{equation}\label{LeProb}
\inf\left\{\int_\O|D w|\;;\;w\in BV(\O)\text{ and }\tr_{\p\O}w=\1_\K\right\}
\end{equation}
is not achieved.
\end{thm}
 %The above theorem has been proved when  $\O=\D$ is the unit disk in \cite{ST1}.  

The calculations in \cite{ST1} are specific to the case $\O=\D$.  The proof of Theorem \ref{MainTHM}  relies on new arguments for the construction of the Cantor set $\K$ and the strategy of the proof.
 
% More precisely in this article we prove a variant of Theorem \ref{MainTHM} specific to the connected set.
% \begin{thm}\label{MainTHMConnex}
%Let $\O\subset\R^2$ be a bounded $C^2$ open connected set. Then there is $\K\subset\p\O$ a  set s.t. the minimization problem
%\begin{equation}\label{LeProb-con}
%\min\left\{\int_\O|D w|\;|\;w\in BV(\O)\text{ and }\tr_{\p\O}w=\1_\K\right\}
%\end{equation} 
%has no solution.
% \end{thm}
 %It is easy to check that Theorem \ref{MainTHM} is a direct consequence of Theorem \ref{MainTHMConnex}.

 \section{Strategy of the proof}\label{SectionStrategy}
 \subsection{The model problem}
 We illustrate the strategy developed to prove Theorem \ref{MainTHM} on the model case  $\Q=(0,1)^2$. Clearly, this model case does not satisfy the $C^2$ assumption. 	
 
 Nevertheless, the flatness of $\p\Q$ allows to get a more general counterpart of Theorem \ref{MainTHM}. Namely, the counterpart of Theorem \ref{MainTHM} [see Proposition \ref{THMLeProb-con-MODEL} below] is no more an existence result of a set $\K\subset\p\Q$ s.t. Problem \eqref{LeProb} is not achieved. It is a non existence result of a least gradient function for $h=\1_\M$ for {\it any} measurable domain $\M\subset[0,1]\times\{0\}\subset\p\Q$ with positive Lebesgue measure.
 
 We thus prove the following result whose strategy of the proof is due to Petru Mironescu.
 \begin{prop}\label{THMLeProb-con-MODEL}[P. Mironescu]
 Let $\tilde{\M}\subset[0,1]$ be a measurable set with positive Lebesgue measure. Then the {\it infimum} in 
 \begin{equation}\label{LeProbCarre}
\inf\left\{\int_\Q|D w|\;;\;w\in BV(\Q)\text{ and }\tr_{\p\Q}w=\1_{\tilde{\M}\times\{0\}}\right\}
\end{equation}
 is not achieved.
% \begin{equation}\label{LeProb-con-MODEL}
%\min\left\{\int_\O|D w|\;|\;w\in BV(\O)\text{ and }\tr_{\p\O}w=\1_\K\right\}
%\end{equation} 
%does not admit a solution.
 \end{prop}
  This section is devoted to the proof of Proposition \ref{THMLeProb-con-MODEL}. We fix a measurable set $\tilde{\M}\subset[0,1]$ with positive measure and  we let  $h=\1_{\tilde{\M}\times\{0\}}$. We argue  by contradiction: we assume that there exists a minimizer $u_0$ of \eqref{LeProbCarre}. We obtain a contradiction  in 3 steps.\\
  
{\bf Step 1.} Upper bound and lower bound\\

This first step consists in obtaining two estimates. The first estimate is the upper bound 
%Using a result in \cite{giusti1984minimal} [Theorem 2.16 $\&$ Remark 2.17] we have
\begin{equation}\label{MainEstimateModel}
\int_{\Q}|D u_0|\leq\|\1_{\tilde{\M}\times\{0\}}\|_{L^1(\p{\Q})}=\H^1(\tilde\M).
\end{equation}
Here, $\H^1(\tilde\M)$ is the length of $\tilde\M$.

Estimate \eqref{MainEstimateModel} follows from Theorem 2.16 and Remark 2.17 in \cite{giusti1984minimal}. Indeed, by combining Theorem 2.16 and Remark 2.17 in \cite{giusti1984minimal}  we may prove that for $h\in L^1(\p\O)$ and for all $\v>0$ there exists a map $u_\v\in BV(\O)$ s.t.
\[
\int_\O|D u_\v|\leq(1+\v)\|h\|_{L^1(\p\O)}\text{ and }\tr_{\p\O}u_\v=h.
\]   The proof of this inequality when $\O$ is a half space  is presented in \cite{giusti1984minimal}. It is easy to adapt  the argument when $\O=\Q=(0,1)^2$. %\footnote{ This result was initially proved by Gagliardo in \cite{gagliardo}. A short proof of "$\tr W^{1,1}=L^1$" may be found in \cite{Mir-Gagliardo}. The key-point is here the quantitative form of the upper bound on $\int_\O|D u_\v|$ given by Remark 2.17 in \cite{giusti1984minimal}.}
The extension for  a $C^2$ set $\O$ is presented in Appendix \ref{AppendixGiusti}.\\% An easy adaptation of the argument may be done for $\O=(0,1)^2$.\\

{\bf Step 2.} Optimality of \eqref{MainEstimateModel} [see \eqref{Satuaration-ModelPb}]\\

The optimality of \eqref{MainEstimateModel} is obtained {\it via} the following lemma.
\begin{lem}\label{TecLem1}
For $u\in BV(\Q)$ we have
\[
\int_{\Q}|D_2 u|\geq\int_0^1|\tr_{\p\Q}  u(\cdot,0)-\tr_{\p\Q} u(\cdot,1)|.
\]
Here, for $k\in\{1,2\}$ we denoted 
\[
\int_\Q|D_k u|=\sup\left\{\int_{\Q} u\p_k\xi\,;\,\xi\in C_c^1({\Q})\text{ and }|\xi|\leq1\right\}
\]
where $C_c^1({\Q})$ are the set of real valued $C^1$-functions with compact support included in $\Q$.
\end{lem}
Lemma \ref{TecLem1} is proved in Appendix \ref{AppProTecLem1}.

From Lemma \ref{TecLem1} we get 
\[
\int_{\Q}|D_2 u_0|\geq\int_0^1|\tr_{\p\Q} u_0(\cdot,0)-\tr_{\p\Q} u_0(\cdot,1)|=\int_0^1\1_{\tilde{\M}\times\{0\}}=\H^1(\tilde{\M}).
\]
Since we have
\begin{eqnarray}\label{Satuaration-ModelPb}
\int_{\Q}|D u_0|&:=&\sup\left\{\int_{\Q} u\Div(\xi)\,;\,\xi=(\xi_1,\xi_2)\in C_c^1({\Q},\R^2)\text{ and }\xi_1^2+\xi_2^2\leq1\right\}
\geq\int_{\Q}|D_2 u_0|
\geq\H^1(\tilde\M),%\int_0^1|\tr_{\p\Q} u_0(\cdot,0)-\tr_{\p\Q} u_0(\cdot,1)|=\int_0^1\1_{\tilde{\M}\times\{0\}}=|\tilde{\M}|.
\end{eqnarray}
we get the optimality of \eqref{MainEstimateModel}. 
\begin{comment}
 We fix $(u_n)_n\subset C^\infty(\O)\cap BV(\O)$ s.t. $\tr_{\p\O}u_n=\1_\K$ and $u_n\tos u_0$. By $u_n\tos u_0$, we mean $(u_n)_n$ strictly converge to $u$, {\it i.e.}, 
 \[
u_n\stackrel{L^1}{\to}u_0\text{ and }\int_{\O}|\n u_n|\to\int_{\O}|D u_0|.
\]
 The sequence is defined by convolution with $u_0$ and a Friedrichs mollifiers with compact support [see \cite{giusti1984minimal}-Theorem 1.17].
 
We let also $B_\infty:=\{(x,y)\in(0,1)^2\,;\,x\in\tilde{\K}\text{ and }y>0\}$. The  saturation of \eqref{MainEstimateModel} consists in obtaining:
\begin{equation}\label{LowerBoundIModel}
\int_{B_\infty}|\p_{2}u_n|\geq|\K|.
\end{equation}

As observed by Mironescu, the key tool in the proof of \eqref{LowerBoundIModel} is Fubini Theorem combined with the Fundamental Theorem of  calculus.
\begin{eqnarray}
\int_{B_\infty}|\p_{2}u_n|&=&\int_0^1{\rm d}x\int_0^1|\p_{2}u_n|\1_{B_\infty}\,{\rm d}y
\\\nonumber
&\geq&\int_0^1{\rm d}x\left|\int_0^1\p_{2}u_n\1_{B_\infty}{\rm d}y\right|
\\\nonumber
&=&\int_0^1\1_{\tilde \K}
\\\label{Satuaration-ModelPb}&=&|\tilde\K|=|\K|.
\end{eqnarray}
\end{comment}

{\bf Step 3.} A transverse argument

From \eqref{MainEstimateModel} and \eqref{Satuaration-ModelPb} we may prove
\begin{equation}\label{NulFirstVar}
\int_{\Q}|D_1 u_0|=0.
\end{equation}
Equality \eqref{NulFirstVar} is a direct consequence of the following lemma.
\begin{lem}\label{TecLem2}
Let $\O$ be a planar open set. If $u\in BV(\O)$ is s.t.
\[
\int_{\O}|D u|=\int_{\O}|D_2 u|,
\]
then $\displaystyle\int_{\O}|D_1 u|=0$.
\end{lem}
Lemma \ref{TecLem2} is proved in Appendix \ref{AppProTecLem2}.

In order to conclude we state an easy lemma.
\begin{lem}\label{TecLem3}[Poincaré inequality]
For $u\in BV(\Q)$ satisfying $\tr_{\p\Q} u=0$ in $\{0\}\times[0,1]$ we have
\[
\int_{\Q}| u|\leq \int_{\Q}|D_1 u|.
\]\end{lem}
Lemma \ref{TecLem3} is proved in Appendix \ref{AppProTecLem3}.

Hence, from \eqref{NulFirstVar} and Lemma \ref{TecLem3} we have $u_0=0$ which is in contradiction with $\tr_{\p\Q} u_0=\1_{\tilde{\M}\times\{0\}}$ with $\H^1(\tilde{\M})>0$.
\subsection{Outline of the proof of Theorem \ref{MainTHM}}
\begin{comment}
To prove Theorem \ref{MainTHM}, we plan to  follow the strategy presented above. The first  main step in this strategy is the construction of a suitable Cantor set $\K\subset\p\O$. Once $\K$ is constructed we define $B_\infty$ by bending  lines orthogonal to $\p\O$.%And then  to {\it bend} the analog of the  connected components of $B_\infty$  for a $C^2$ domain.% instead the square $(0,1)^2$. 

The construction of $\K$ is done Section \ref{SectConstructCantor} and $B_\infty$ is defined Section \ref{SectionConstructFunction}.\end{comment}
The idea is to adapt the above construction and argument to the case of a general $C^2$ domain $\O$. If $\O$ has a flat or concave part $\Gamma$ of the boundary $\p\O$, then a rather straightforward variant of the above proof shows that $\1_\M$, where $\M$ is a non trivial part of $\Gamma$, is not the trace of a least gradient function.

\begin{remark}Things are more involved when $\O$ is convex. For simplicity we illustrate this fact when $\O=\mathbb{D}=\{x\in\R^2\,;\,|x|<1\}$. Let $\M\subset\S^1\cap\{(x,y)\in\R^2\,;\,x<0\}$ be an arc whose endpoints are symmetric with respect to the $x$-axis. We let $(x_0,-y_0)$ and $(x_0,y_0)$ be the endpoints of $\M$ [here $x_0\leq0$ and $y_0>0$].

We let $\C$ be the chord of $\M$. On the one hand, if $u\in C^1(\mathbb{D})\cap W^{1,1}(\mathbb{D})$ is s.t. $\tr_{\S^1} u=\1_\M$ then, using the Fundamental Theorem of calculus, we have for $-y_0<y<y_0$
\[
\int_{-\sqrt{1-y^2}}^{\sqrt{1-y^2}}|\p_x u(x,y)|\geq1.
\]
Thus we easily get 
\[
\int_{\mathbb{D}}|\n u|\geq \int_{\mathbb{D}}|\p_x u|\geq\int_{-y_0}^{y_0}{\rm d}y\int_{-\sqrt{1-y^2}}^{\sqrt{1-y^2}}|\p_x u(x,y)|\geq2y_0=\H^1(\C).
\]
Consequently, with the help of a density argument [{\it e.g.} Lemma \ref{LemApproxTechskdjfgh} in Appendix \ref{SLemApproxTechskdjfgh}] we obtain 
\[
\inf\left\{\int_{\mathbb{D}}|D u|\;;\;u\in BV(\mathbb{D})\text{ and }\tr_{\S^1}u=\1_{{\M}}\right\}\geq\H^1(\C).
\]
On the other hand we let $\o:=\{(x,y)\in\R^2\,;\,x<x_0\}$. It is clear that $u_0=\1_{\o}\in BV(\mathbb{D})$ and $\tr_{\S^1}u_0=\1_{{\M}}$. Moreover
\[
\int_{\mathbb{D}}|D u_0|=\H^1(\C).
\]
Consequently $u_0$ is a  function of least gradient. We may do the same argument for a domain $\O$ as soon as we have a chord entirely contained in $\O$. This example suggest that for a convex set $\O$, the construction of a  set $\K\subset\p\O$ s.t. \eqref{LeProb} is not achieved has to be "sophisticated".
\end{remark}
The strategy to prove Theorem \ref{MainTHM} consists of constructing a special set $\K\subset\p\O$ [of Cantor type] and to associate to $\K$ a set $B_\infty$ [the analog of $\tilde{\M}\times(0,1)$ in the model problem] which "projects" onto $\K$ and s.t., if $u_0$ is a minimizer of  \eqref{LeastGradH}, then
\begin{equation}\label{MainStraEst1}
\int_{B_\infty}|\vec{X}\cdot Du_0|\geq\H^1(\K).
\end{equation}
Here, $\vec{X}$ is a vector field satisfying $|\vec{X}|\leq1$. It is the curved analog of $\vec{X}={\bf e}_2$ used in the above proof.

By \eqref{MainStraEst1} [and Proposition \ref{ExtGiustiUpBound} in Appendix \ref{AppendixGiusti}], if $u_0$ is a minimizer, then 
\begin{equation}\label{MainStraEst2}
\int_{\O\setminus B_\infty}|D u_0|+\int_{B_\infty}(|D u_0|-|\vec{X}\cdot Du_0|)=0.
\end{equation}
We next establish a Poincaré type inequality implying that any $u_0$ satisfying \eqref{MainStraEst2} and $\tr_{\p\O\setminus\K} u=0$ is $0$, which is not possible.

The heart of the proof consists of constructing $\K, B_\infty$ and $\vec{X}$ [see Sections \ref{SectConstructCantor} and \ref{SectionConstructFunction}].
\section{Notation, definitions}\label{NotationOfTheArticle}
The ambient space is the Euclidean plan $\R^2$. We let $\mathcal{B}_{\rm can}$ be the canonical basis of $\R^2$.
\begin{enumerate}[a)]
\item The open ball centered at $A\in\R^2$ with radius $r>0$ is denoted $B(A,r)$.
\item A vector may be denoted by an arrow when it is defined by its endpoints ({\it e.g.} $\vecc{AB}$).  It may be also denoted by a letter  in bold font ({\it e.g.} ${\bf u}$) or more simply by a Greek letter in normal font ({\it e.g. } $\nu$). \\We let also $|{\bf u}|$ be the  Euclidean norm of the vector ${\bf u}$.
\item For a vector ${\bf u}$ we let ${\bf u}^\bot$ be the direct orthogonal vector to ${\bf u}$, {\it i.e.}, if  ${\bf u}=(x_1,x_2)$ then  ${\bf u}^\bot=(-x_2,x_1)$.
\item For $A,B\in\R^2$, the segment of endpoints $A$ and $B$ is denoted $[AB]=\{A+t\vecc{AB}\,;\,t\in[0,1]\}$ and $\dist(A,B)=|\vecc{AB}|$ is the Euclidean distance.
\item For a set $U\subset\R^2$, the topological interior of $U$ is denoted by $\stackrel{\circ}{U}$ and its topological closure is  $\overline{U}$.

\item For $k\geq1$, a  $C^k$-curve is the range of a $C^k$ injective map from $(0,1)$ to $\R^2$. Note that, in this article, $C^k$-curves are not closed sets of $\R^2$.
\item  For $\Gamma$ a $C^1$-curve, $\H^1(\Gamma)$ is the $1$-dimensional Hausdorff measure of $\Gamma$. %Note that  we have $\H^1(\Gamma)<\infty$.
\item  For $k\geq1$, a $C^k$-Jordan curve is the range of a $C^k$ injective map from the unit circle $\S^1$ to $\R^2$.
\item For  $\Gamma$  a $C^1$-curve or a $C^1$-Jordan curve, $\C=[AB]$ is a chord of $\Gamma$ when $A,B\in\overline{\Gamma}$ with $A\neq B$. %Here $\overline{\Gamma}$ is the closure of $\Gamma$.
\item\label{DefinitionEtaGamma} If $\Gamma$ is a $C^1$-Jordan curve then, for $A, B\in\Gamma\,\&\,A\neq B$, the set $\Gamma\setminus\{A,B\}$ admits exactly two connected components: $\Gamma_1\&\Gamma_2$. These connected components are $C^1$-curves.

By smoothness of $\Gamma$, it is clear that there exists $\eta_\Gamma>0$ s.t. for $0<\dist(A,B)<\eta_\Gamma$ there exists THE smallest connected components: we have $\H^1(\Gamma_1)<\H^1(\Gamma_2)$ or $\H^1(\Gamma_2)<\H^1(\Gamma_1)$.

If $0<\dist(A,B)<\eta_\Gamma$ we may define $\curve{AB}$ by:
\begin{equation}\label{DefCurve}
\text{$\curve{AB}$ is the closure of the smallest curve between $\Gamma_1$ and $\Gamma_2$.}
\end{equation}
\item In this article $\O\subset\R^2$ is a  $C^2$ bounded open set. \begin{comment}By $C^2$ we mean that $\p\O$ is of class $C^2$:

\item There exists a covering $\{U_1,...,U_n\}$ of $\p\O$ by open sets, $\p\O\subset\cup_{i=1}^nU_i$, 
\item For all $i\in\{1,...,n\}$ there exists a $C^2$ diffeomorphism $\varphi_i:\overline{U_i}\to\overline{\D}$ s.t.
\[
\left|\begin{array}{c}\varphi_i(U_i\cap\O)=\{(x_1,x_2)\in\D\,;\,x_2>0\}\\\varphi_i(U_i\cap\p\O)=\{(x_1,x_2)\in\D\,;\,x_2=0\}\end{array}\right..
\]
\end{itemize}\end{comment}
%\item\label{UseOfRad} For a non empty set $A\subset\R^2$ we define its radius by 
%\[
%{\rm rad}(A)=\sup\{r\geq0\,;\,\exists\,x\in A \text{ s.t. } \overline{B(x,r)}\subset A\}.
%\]
%Note that the topological interior of $A$ is empty if and only if ${\rm rad}(A)=0$.
\end{enumerate}
 
\section{Construction of the Cantor set $\K$}\label{SectConstructCantor}
It is clear that, in order to prove Theorem \ref{MainTHM}, we may assume that $\O$ is a connected set.

We fix $\O\subset\R^2$ a bounded $C^2$ open connected set. The set $\K\subset\p\O$  is a Cantor type set we will construct below.%. It is constructed by a recursive way.

\subsection{First step:  localization of $\p\O$}

From the regularity of $\O$, there exist $\ell+1$ $C^2$-open sets, $\o_0,...,\o_\ell$, s.t. $\O=\o_0\setminus\overline{\o_1\cup\cdots\cup\o_\ell}$ and %Moreover we may assume that:
\begin{itemize}
\item $\o_i$ is simply connected for $i=0,...,\ell$,
\item $\overline{\o_i}\subset\o_0$ for $i=1,...,\ell$,
\item $\overline{\o_i}\cap\overline{\o_j}=\emptyset$ for $1\leq i<j\leq \ell$.
\end{itemize}

We let $\Gamma=\p\o_0$. The Cantor type set $\K$ we construct "lives" on $\Gamma$. Note that $\Gamma$ is a Jordan-curve. 
 
 Let $M_0\in\Gamma$ be s.t. the inner curvature of $\Gamma$ at $M_0$ is positive [the existence of $M_0$ follows from the Gauss-Bonnet formula]. Then there exists  $r_0\in(0,1)$ s.t.  $[AB]\subset\overline{\O}$ and $[AB]\cap\p\O=\{A,B\}$, $\forall\,A,B\in B(M_0,r_0)\cap\Gamma$. Note that we may assume $2r_0<\eta_\Gamma$ [$\eta_\Gamma$ is defined in Section \ref{NotationOfTheArticle}-\ref{DefinitionEtaGamma}].\\

We fix $ A, B\in B(M_0,r_0)\cap\Gamma$ s.t.  $A\neq B$. We have:
\begin{itemize}
\item  By the definition of $M_0$ and $r_0$, the chord $\C_0:=[ A B]$ is included in $\overline{\O}$.
\item  We let $\curve{AB}$ be the closure of the smallest part of $\Gamma$ which is  delimited by $A,B$ (see \eqref{DefCurve}). We may assume that $\curve{AB}$ is the graph of $f\in C^2([0,\eta],\R^+)$ in the orthonormal frame $\repereO=( A,{\bf e}_1,{\bf e}_2)$ where ${\bf e}_1={\vecc{ A B}}/{|\vecc{ A B}|}$.
\item The function $f$ satisfies $f(x)>0$ for $x\in(0,\eta)$ and $f''(x)<0$  for $x\in[0,\eta]$.
\end{itemize}
For further use we note that the length of the chord $[ A B]$ is $\eta$ and that for intervals $I,J\subset[0,\eta]$, if $I\subset J$ then 
\begin{equation}\label{KeyArgRestriction}
\begin{cases}
\|f_{|I}'\|_{L^\infty(I)}\leq\|f_{|J}'\|_{L^\infty(J)}
\\
\|f_{|I}''\|_{L^\infty(I)}\leq\|f_{|J}''\|_{L^\infty(J)}
\end{cases}
\end{equation}
where $f_{|I}$ is the restriction of $f$ to $I$.\\
Replacing the chord $\C_0=[ A B]$ with a smaller chord of $\curve{AB}$ parallel to $\C_0$, %and up to restrict the function $f$ (and up to add a suitable constant to $f$) 
we may assume that
\begin{equation}\label{MainGeomHyp}
0<\eta<\min\left\{\frac{1}{2}\,;\,\dfrac{1}{16\|f''\|_{L^\infty([0,\eta])}^2}\,;\,\dfrac{1}{2\|f'\|_{L^\infty([0,\eta])}\|f''\|_{L^\infty([0,\eta])}}\right\}.
\end{equation}
We may also assume that
\begin{itemize}
\item Letting  $D^+_0$ be the bounded open set s.t. $\p D_0^+=[AB]\cup\curve{AB}$ we have $\Pi_{\p\O}$, the orthogonal projection on $\p\O$, is well defined and of class $C^1$ in $D^+_0$.
%\item The map $\Pi_0:[AB]\to\curve{AB}$, $x\mapsto\Pi_{\p\O}(x)$ is bijective.
\item We have
 \begin{equation}\label{AltFondEq}
 1+4\|f''\|^2_{L^\infty}{\rm diam}(D_0^+)<\dfrac{16}{9}
\end{equation}
where ${\rm diam}(D_0^+)=\sup\{\dist(M,N)\,;\,M,N\in D_0^+\}$. [Here we used \eqref{KeyArgRestriction}]
\end{itemize}
\subsection{Step 2: Iterative construction}\label{SectHerdityCOnstruKantor}
We are now in position to construct the Cantor type set $\K$ as a subset of $\curve{AB}$. The construction is iterative. 

The goal of the construction is to get at step $N\geq0$ a collection of $2^N$ pairwise disjoint curves included in $\curve{AB}$ [denoted by $\{K^N_1,...,K^N_{2^N}\}$] and their chords [denoted by $\{\C^N_1,...,\C^N_{2^N}\}$].

The idea is standard: at the step $N\geq0$ we replace a curve $ \Gamma_0$ included in $\curve{AB}$ by two curves included in $\Gamma_0$ (see Figure \ref{Heredite}).
\\{\bf Initialization.} We initialize the procedure by letting $K_1^0:=\curve{AB}$ and $\C_1^0=\C_0=[AB]$. 
 \\
 
 At step $N\geq0$ we have: 
 \begin{itemize}
 \item A set of $2^N$  curves included in $\curve{AB}$, $\{K_1^N,...,K_{2^N}^N\}$. The curves $K_k^N$'s are mutually disjoint. We let $\K_N=\cup_{k=1}^{2^N}K_k^N$.
 \item A set of $2^N$  chords, $\{\C_1^N,...,\C_{2^N}^N\}$ s.t. for $k=1,...,2^N$, $\C_k^N$ is the chord of $K_k^N$.
 \end{itemize}
 \begin{remark}\label{ConstructionFact}
 \begin{enumerate}
 \item\label{RemCorde} Note that since  the $\C_k^N$'s are chords of $\curve{AB}$ and since in the frame $\repereO=( A,{\bf e}_1,{\bf e}_2)$, $\curve{AB}$ is the graph of a function, none of the chords $\C_k^N$ is vertical, {\it i.e.}, directed by  ${\bf e}_2$.

Since the chords $\C_k^N$ are not vertical, for $k\in\{1,...,2^N\}$,  we may define $\nu_{\C_k^N}$ as the unit vector orthogonal to $\C_k^N$ s.t. $\nu_{\C_k^N}=\alpha {\bf e}_1+\beta {\bf e}_2$ with $\beta>0$. 
 \item\label{ConstrFact2} For $\eta$ satisfying \eqref{MainGeomHyp}, % Since the function  $f\in C^2([0,\eta],\R^+)$  is strictly concave, 
if we consider a chord $\C_k^N$ and a straight line $D$ orthogonal to $\C_k^N$ and intersecting $\C_k^N$, then the straight line $D$ intersect $K_k^N$ at exactly one points. This fact is proved in Appendix \ref{ProofOfChianteRemark}.
 
\begin{comment} Indeed, let  $f\in C^2([0,\eta],\R^+)$  be a  strictly concave and let $A=(a,f(a)),B=(b,f(b))$ with $0\leq a<b\leq \eta$. For $M\in[AB]$, let $D_M$ be the straight line orthogonal to $\curve{AB}$. Here $\curve{AB}$ is the connected part of the graph of $f$ with endpoints $A$ and $B$.
 
 The existence of $N\in\curve{AB}\cap D_M$ is clear. We just prove uniqueness. %since the $\Pi_{(AB)}$, the orthogonal projection on the line $(AB)$ is continuous. Moreover its restriction to strip $[AB]+\R\nu_{[AB]}$ admits $[AB]$ for range. Let $c\in[a,b]$ be s.t. $f'(c)=\dfrac{f(b)-f(a)}{b-a}$ and denote $\tilde{D}$ be tangent of the graph of $f$ at $(c,f(c))$. Since $\tilde{D}$ is parallel to $[AB]$, we may consider, $\mathcal{R}$, be the rectangle delimited by $[AB]$, $\tilde{D}$ and  the lines $A+\R\nu_{[AB]}$ and $B+\R\nu_{[AB]}$. From strict concavity of $f$, the graph of $f_{[a,b]}$ is included in $\R$
 
 Assume that there existe $N\neq N'$ s.t. $N,N'\in\curve{AB}\cap D_M$. We denote $M=(x,Then $[NN']$ is a chord of $\curve{AB}$
 This fact is proved in \textcolor{red}{Appendix \ref{AppPreuveFat}}.
 \end{comment}
 %may be easily checked by letting $\Gamma^*$ be the range of $\curve{AB}$ with respect to the symmetry of axe $[AB]$ and by noting that the compact set delimited by $\curve{AB}\cup\Gamma^*$ is a strictly convex domain.
 \end{enumerate}
 \end{remark}
{\bf Induction rules.} From step $N\geq0$ to step $N+1$ we follow the following rules:
\begin{enumerate}
\item For each $k\in\{1,...,2^N\}$, we let $\eta_k^N$ be the length of  $\C_k^N$. Inside the chord $\C_k^N$ we center a segment $I_k^N$ of length $(\eta_k^N)^2$. 
\item With the help of Remark \ref{ConstructionFact}.2, we may define two distinct points of $K_k^N$ as the intersection of $K_k^N$ with straight lines orthogonal to $\C_k^N$ which pass  to the endpoints of $I_k^N$. 
\item These intersection points are the endpoints of a curve $\tilde K_k^N$ included in $K_k^N$.  We let $K_{2k-1}^{N+1}$ and $K_{2k}^{N+1}$ be the connected components of $ K_k^N\setminus \overline{\tilde  K_k^N}$. We let also \begin{itemize}
 \item $\C_{2k-1}^{N+1}$ and $\C_{2k}^{N+1}$ be the corresponding chords;
 \item $\K_{N+1}=\cup_{k=1}^{2^{N+1}}K_{k}^{N+1}$.
 \end{itemize}
\end{enumerate}
\begin{notation}\label{Not.FatherChord}
A natural terminology consists in defining the father and the sons of a chord or  a curve:
\begin{itemize}
\item  $\mathcal{F}(\C_{2k-1}^{N+1})=\mathcal{F}(\C_{2k}^{N+1})=\C_{k}^{N}$ is the {\it father} of the chords $\C_{2k-1}^{N+1}$ and $\C_{2k}^{N+1}$.

$\mathcal{F}(K_{2k-1}^{N+1})=\mathcal{F}(K_{2k}^{N+1})=K_{k}^{N}$ is the {\it father} of the curves $K_{2k-1}^{N+1}$ and $K_{2k}^{N+1}$.
\item $\mathcal{S}(\C_{k}^{N})=\{\C_{2k-1}^{N+1},\C_{2k}^{N+1}\}$ is the set of sons of the chord $\C_{k}^{N}$, {\it i.e.}, $\mathcal{F}(\C_{2k-1}^{N+1})=\mathcal{F}(\C_{2k}^{N+1})=\C_{k}^{N}$.

$\mathcal{S}(K_{k}^{N})=\{K_{2k-1}^{N+1},K_{2k}^{N+1}\}$ is the set of sons of the curve $K_{k}^{N}$, {\it i.e.}, $\mathcal{F}(K_{2k-1}^{N+1})=\mathcal{F}(K_{2k}^{N+1})=K_{k}^{N}$.
\end{itemize}
\end{notation}

The inductive procedure is represented in Figure \ref{Heredite}.

\begin{figure}[H]
\psset{xunit=.70cm,yunit=.70cm,algebraic=true,dimen=middle,dotstyle=o,dotsize=3pt 0,linewidth=0.8pt,arrowsize=3pt 2,arrowinset=0.25}
\begin{pspicture*}(-0.5,-2)(25,3)
\parametricplot[linewidth=3.6pt]{0.7853981633974483}{2.356194490192345}{1*5.66*cos(t)+0*5.66*sin(t)+4|0*5.66*cos(t)+1*5.66*sin(t)+-4}
\psline(0,0)(8,0)
\psline(9.5,0.5)(9.5,-0.5)
\psline(12,-0.5)(9.5,-0.5)
\psline(9.5,0.5)(12,0.5)
\psline(12,1)(12,-1)
\psline(12,-1)(13.5,0)
\psline(12,1)(13.5,0)
\parametricplot{0.7853981633974483}{2.356194490192345}{1*5.66*cos(t)+0*5.66*sin(t)+20|0*5.66*cos(t)+1*5.66*sin(t)+-4}
\parametricplot[linewidth=3.6pt]{0.7853981633974483}{1.5000370279717954}{1*5.66*cos(t)+0*5.66*sin(t)+20|0*5.66*cos(t)+1*5.66*sin(t)+-4}
\parametricplot[linewidth=3.6pt]{1.6433575184321727}{2.356194490192345}{1*5.66*cos(t)+0*5.66*sin(t)+20|0*5.66*cos(t)+1*5.66*sin(t)+-4}
\psline(19.59,0)(19.59,1.7)
\psline(20.4,0)(20.4,1.7)
\psline(16,0)(24,0)
\psline{|-|}(19.59,-0.2)(20.4,-0.2)
\rput(20,-.6){$(\eta_N^k)^2$}
\psline{|-|}(16,-1)(24,-1)
\rput(20,-1.3){$\eta_N^k$}
\psline(20.4,0)(20.5,0)
\psline(20.5,0.1)(20.4,0.1)
\psline(20.5,0.1)(20.5,0)
\psline(19.59,0.1)(19.48,0.1)
\psline(19.48,0.1)(19.48,0)
\end{pspicture*}
\caption{Induction step}\label{Heredite}
\end{figure}
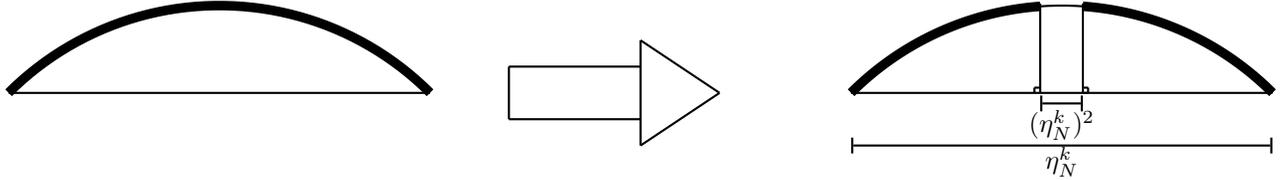

In Figure \ref{first}$\&$\ref{second} the two first iterations of the process are represented.
\begin{figure}[h] \begin{minipage}[b]{0.49\linewidth}\centering
\newrgbcolor{xdxdff}{0.49 0.49 1}
\psset{xunit=.45cm,yunit=.45cm,algebraic=true,dimen=middle,dotstyle=o,dotsize=3pt 0,linewidth=0.8pt,arrowsize=3pt 2,arrowinset=0.25}
\begin{pspicture*}(2.7,-1)(21,5)
\parametricplot{0.6035949011944017}{2.5379977523953916}{1*11.03*cos(t)+0*11.03*sin(t)+11.92|0*11.03*cos(t)+1*11.03*sin(t)+-6.26}
\psline(2.84,0)(21,0)
\psline(10.98,0)(10.98,4.8)
\psline(12.86,0)(12.86,4.8)
%\psline(2.84,0)(10.98,4.73)
%\psline(12.86,4.73)(21,0)
\parametricplot[linewidth=4pt]{1.6561314031258352}{2.5379977523953916}{1*11.03*cos(t)+0*11.03*sin(t)+11.92|0*11.03*cos(t)+1*11.03*sin(t)+-6.26}
\parametricplot[linewidth=3.6pt]{0.6035949011944017}{1.4854612504639584}{1*11.03*cos(t)+0*11.03*sin(t)+11.92|0*11.03*cos(t)+1*11.03*sin(t)+-6.26}
\begin{scriptsize}
%\psdots[dotstyle=*,linecolor=blue](11.92,-6.26)
%\rput[bl](12,-6.14){\blue{$A$}}
%\psdots[dotstyle=*,linecolor=blue](21,0)
%\rput[bl](21.08,0.12){\blue{$B$}}
%\psdots[dotstyle=*,linecolor=xdxdff](2.84,0)
%\rput[bl](2.92,0.12){\xdxdff{$C$}}
%\rput[bl](12.04,4.9){$c$}
%\rput[bl](11.92,-0.5){\normalsize$\eta_1^2$}
%\psdots[dotstyle=*,linecolor=xdxdff](10.98,0)
%\rput[bl](11.06,0.12){\xdxdff{$D$}}
%\psdots[dotstyle=*,linecolor=darkgray](12.86,0)
%\rput[bl](12.94,0.12){\darkgray{$E$}}
\rput[bl](11.14,6){$b$}
\rput[bl](13,6){$e$}

%\psdots[dotstyle=*,linecolor=darkgray](17.7,3.13)
%\rput[bl](17.78,3.26){\darkgray{$O$}}
%\rput[bl](6.5,3.4){$q$}
%\rput[bl](17.56,3.4){$r$}
\end{scriptsize}
\end{pspicture*}\caption{First iteration of the process}\label{first}
\end{minipage}
 \begin{minipage}[b]{0.49\linewidth}\centering
\newrgbcolor{xdxdff}{0.49 0.49 1}
\psset{xunit=.45cm,yunit=.45cm,algebraic=true,dimen=middle,dotstyle=o,dotsize=3pt 0,linewidth=0.8pt,arrowsize=3pt 2,arrowinset=0.25}
\begin{pspicture*}(1.5,-1)(22,5)
\parametricplot{0.6035949011944017}{2.5379977523953916}{1*11.03*cos(t)+0*11.03*sin(t)+11.92|0*11.03*cos(t)+1*11.03*sin(t)+-6.26}
\psline(2.84,0)(21,0)
\psline(2.84,0)(10.98,4.73)
\psline(12.86,4.73)(21,0)
\parametricplot{1.6561314031258352}{2.5379977523953916}{1*11.03*cos(t)+0*11.03*sin(t)+11.92|0*11.03*cos(t)+1*11.03*sin(t)+-6.26}
\parametricplot[linewidth=3.6pt]{2.1221806941009618}{2.5379977523953916}{1*11.03*cos(t)+0*11.03*sin(t)+11.92|0*11.03*cos(t)+1*11.03*sin(t)+-6.26}
\parametricplot[linewidth=3.6pt]{1.6561314031258352}{2.071948461420265}{1*11.03*cos(t)+0*11.03*sin(t)+11.92|0*11.03*cos(t)+1*11.03*sin(t)+-6.26}
\parametricplot[linewidth=3.6pt]{0.6035949011944017}{1.0194119594888322}{1*11.03*cos(t)+0*11.03*sin(t)+11.92|0*11.03*cos(t)+1*11.03*sin(t)+-6.26}
\parametricplot[linewidth=3.6pt]{1.0696441921695303}{1.4854612504639584}{1*11.03*cos(t)+0*11.03*sin(t)+11.92|0*11.03*cos(t)+1*11.03*sin(t)+-6.26}
\psplot{6.15}{6.655}{(-64.82--8.14*x)/-4.73}
\psplot{6.6}{7.17}{(-70.03--8.14*x)/-4.73}
\psplot{16.7}{17.25}{(-124.02--8.14*x)/4.73}
\psplot{17.15}{17.75}{(-129.24--8.14*x)/4.73}
\begin{scriptsize}

%\rput[bl](7.06,2.1){$\eta_2^2$}
%\rput[bl](16.76,2.1){$\eta_2^2$}

%\psdots[dotstyle=*,linecolor=darkgray](17.7,3.13)
%\rput[bl](17.78,3.26){\darkgray{$O$}}

\end{scriptsize}
\end{pspicture*}\caption{Second iteration of the process}\label{second}\end{minipage}
\end{figure}

We now define the  Cantor type set 
\begin{equation}\label{defKinffty}
\K=\bigcap_{N\geq0} \overline{\K_N}.
\end{equation}
The Cantor type set $\K$ is fat:
\begin{prop}
We have $\H^1(\K)>0$.
\end{prop}
This proposition is proved in Appendix \ref{ProofOfPropDefKinkdfg}.
\section{Construction of a sequence of functions}\label{SectionConstructFunction}% used in conjuction with Coarea Formula}
A key argument in the proof of Theorem \ref{MainTHM} is the use of the coarea formula to calculate a lower bound for \eqref{LeProb}. The coarea formula is applied to a function adapted to the set $\K$.

For $N=0$ we let
\begin{itemize}
\item[$\bullet$] $D_0^+$ be the compact set delimited by $K_0=\curve{AB}$ and $\C_1^0:=[AB]$ the chord of $K_0$.
\item[$\bullet$] We  recall that we fixed a frame $\repereO=( A,{\bf e}_1,{\bf e}_2)$ where ${\bf e}_1={\vecc{ A B}}/{|\vecc{ A B}|}$.  For $\sigma=(\sigma_1,0)\in \C_1^0$, we define:
\begin{equation}\label{defISigma}
\text{${I}_\sigma$ is the connected component of $\{(\sigma_1,t)\in\O\,;\,t\leq0\}$ which contains $\sigma$.}
\end{equation} [${I}_\sigma$ is a vertical segment included in $\O$]. 
\item[$\bullet$] $D_0^-=\cup_{\sigma\in{\C_1^0}}I_ \sigma$.
\item[$\bullet$]We now define the maps
\[
\begin{array}{cccc}
\tilde\Psi_0:&D_0^-&\to&\C_1^0\\&x&\mapsto&\Pi_{\C_1^0}(x)
\end{array}
\]
and
\[
\begin{array}{cccc}
\Psi_0:&D_0^-\cup D_0^+&\to&\C_1^0\\&x&\mapsto&\begin{cases}\Pi_{\p\O}(x)&\text{if }x\in D_0^+\\\Pi_{\p\O}[\tilde\Psi_0(x)] &\text{if }x\in D_0^-\end{cases}

\end{array}
\]
where $\Pi_{\p\O}$ is the orthogonal projection on $\p\O$ and $\Pi_{\C_1^0}$ is the orthogonal projection on $\C_1^0$. Note that, in the frame $\repereO$, for $x=(x_1,x_2)\in D_0^-$, we have $\Pi_{\C_1^0}(x)=(x_1,0)$.
\end{itemize}
For $N=1$ and $k\in \{1,2\}$ we let:
\begin{itemize}
\item $D_{k}^1$ be the compact set delimited by  $K_k^1$ and $\C_k^1$ ;
\item $T_k^1$ be the compact right-angled triangle (with its interior) having $\C_k^1$ as side adjacent to the right angle and whose hypothenuse is included in $\C_1^0$;

%\textcolor{red}{\bf Il est fondamental de remarquer que par l'Hypothèse \eqref{MainGeomHyp} et la Proposition \ref{PropMainGeomHyp} les hypoténuses des triangles $T_1^1$ et $T_2^1$ sont strictement disjoints.}
\item $H_k^1$ be the hypothenuse of $T_k^1$. 
\end{itemize}
We now define $D_1^-=\tilde{\Psi}_0^{-1}(H_1^1\cup H_2^1)$, $T_1=T_1^1\cup T_2^1$ and $D_1^+=D_1^1\cup D_2^1$. %Note that $D_1^-,T_1$ and $D_1^+$ are compact sets.

 We first consider the map
\[
\begin{array}{cccc}
\tilde\Psi_1:&T_1\cup D_1^-&\to&\C_1^1\cup\C_2^1\\&x&\mapsto&\begin{cases}\Pi_{\C_k^1}(x)&\text{if }x\in T_k^1\\\Pi_{\C_k^1}[\tilde\Psi_0(x)] &\text{if }x\in  D_1^-\end{cases}.
\end{array}
\]
In Appendix \ref{AppeSepTr} [Lemma \ref{PropMainGeomHyp} and Remark \ref{RemPropMainGeomHyp}], it is proved  that  the triangles $T_1^1$ and $T^1_2$ are disjoint. Thus the map  $\tilde\Psi_{1}$ is well defined

By projecting $\C_1^1\cup\C_2^1$ on $\p\O$ we get
\[
\begin{array}{cccc}
\Psi_1:&T_1\cup D_1^-\cup D_1^+&\to&\K_1\\&x&\mapsto&\begin{cases}\Pi_{\p\O}(x)&\text{if }x\in D_1^+\\\Pi_{\p\O}[\tilde\Psi_1(x)] &\text{if }x\in T_1\cup D_1^-\end{cases}
\end{array}.
\]
%\item[$\bullet$] 

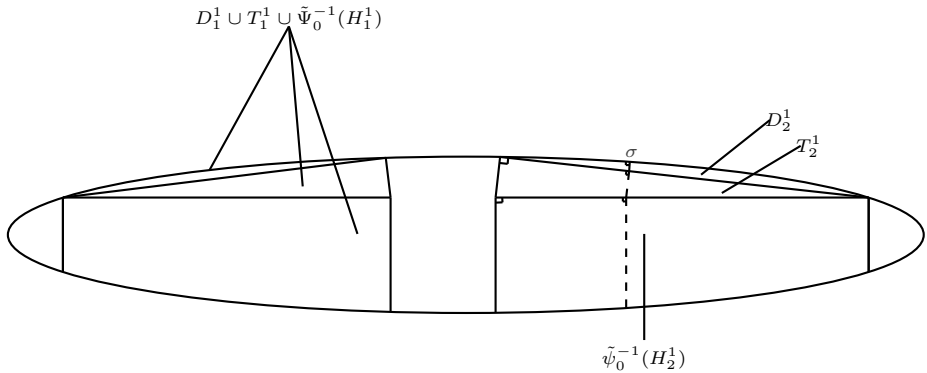
\begin{figure}[H]\begin{center}
\psset{xunit=.400cm,yunit=.40cm,algebraic=true,dimen=middle,dotstyle=o,dotsize=3pt 0,linewidth=0.8pt,arrowsize=3pt 2,arrowinset=0.25}
\begin{pspicture*}(-1,-5.4)(31,7.8)
\rput{0}(15,0){\psellipse(0,0)(15.22,2.6)}
\psline(1.61,1.24)(12.34,2.56)
\psline(16.14,2.59)(28.39,1.24)
\psline(12.34,2.56)(12.5,1.24)
\psline(15.99,1.24)(16.14,2.59)
\psline(1.61,1.24)(1.61,-1.24)
\psline(12.5,1.24)(12.5,-2.56)
\psline(15.99,1.24)(15.99,-2.59)
\psline(28.39,1.24)(28.39,-1.24)
\psline[linestyle=dashed,dash=3pt 3pt](20.33,1.24)(20.43,2.12)
\psline[linestyle=dashed,dash=3pt 3pt](20.33,1.24)(20.33,-2.43)
\psline[linestyle=dashed,dash=3pt 3pt](20.45,2.43)(20.43,2.12)
\psline(16.12,2.39)(16.39,2.36)
\psline(16.39,2.36)(16.41,2.56)
\psline(20.32,2.32)(20.32,2.43)
\psline(20.32,2.32)(20.45,2.32)
\psline(20.34,2.13)(20.32,2)
\psline[linestyle=dashed,dash=3pt 3pt](20.32,2)(20.42,2)
\psline(20.22,1.24)(20.22,1.11)
\psline(20.22,1.11)(20.33,1.11)
\psline(15.99,1.07)(16.22,1.08)
\psline(16.22,1.08)(16.21,1.24)
\psline(1.61,1.24)(12.5,1.24)
\psline(15.99,1.24)(28.39,1.24)
\psline(1.61,1.24)(1.61,-1.24)
\psline(28.39,1.24)(28.39,-1.24)
\psline(22.82,2)(25.11,3.82)
\psline(23.51,1.39)(26.13,2.96)
\psline(20.93,0.03)(20.93,-3.5)
\psline(6.48,2.15)(9.12,6.93)
\psline(9.58,1.61)(9.12,6.93)
\psline(11.4,0.03)(9.12,6.93)
\begin{scriptsize}
%\rput(20.45,2.43){}
\rput[bl](20.3,2.6){\darkgray{$\sigma$}}
\rput(25.4,3.82){$D_2^1$}
\rput(26.4,2.96){$T_2^1$}
\rput(20.93,-4.1){$\tilde{\psi}_0^{-1}(H^1_2)$}
\rput(9.12,7.2){$D_1^1\cup T^1_1\cup\tilde{\Psi}_0^{-1}(H^1_1)$}
\end{scriptsize}
\end{pspicture*}\caption{The sets defined at Step $N=1$ and the dashed level line  of $\Psi_1$ associated to $\sigma\in\K_1$}\end{center}
\end{figure}

For $N\geq1$, we first construct $\tilde\Psi_{N+1}$ and then $\Psi_{N+1}$ is obtained from $\tilde\Psi_{N+1}$ and $\Pi_{\p\O}$.

For $k\in\{1,...,2^{N+1}\}$, we let
\begin{itemize}%[$\bullet$]  
\item $D^{N+1}_k$ be the compact set delimited by $K^{N+1}_k$ and $\C_k^{N+1}$ [recall that $\C_k^{N+1}$ is the chord associated to $K^{N+1}_k$] ;
\item $T_k^{N+1}$ be the right-angled triangle (with its interior) having $\C_k^{N+1}$ as side adjacent to the right angle and whose hypothenuse is included in $\mathcal{F}(\C_k^{N+1})$.  Here $\mathcal{F}(\C_k^{N+1})$ is the father of $\C_k^{N+1}$ (see Notation \ref{Not.FatherChord});
\item $H_k^{N+1}\subset\mathcal{F}(\C_k^{N+1})$  be the hypothenuse of $T_k^{N+1}$.
\end{itemize}
 We denote  $T_{N+1}=\displaystyle\bigcup_{k=1}^{2^{N+1}}T_k^{N+1}$, $D_{N+1}^-=\tilde{\Psi}_{N}^{-1}\left(\displaystyle\bigcup_{k=1}^{2^{N+1}}H_k^{N+1}\right)$ and $D_{N+1}^+=\displaystyle\bigcup_{k=1}^{2^{N+1}}D_k^{N+1}$.% et $\tilde{T}_{N+1}=\cup_{l=1}^{N}T_l$. 
 \begin{figure}[h!]\begin{center}
\psset{xunit=.10cm,yunit=.10cm,algebraic=true,dimen=middle,dotstyle=o,dotsize=3pt 0,linewidth=0.8pt,arrowsize=3pt 2,arrowinset=0.25}
\begin{pspicture*}(-18,-40)(72.5,25)
\parametricplot{0.6900067105221149}{2.4407382457053424}{1*57.41*cos(t)+0*57.41*sin(t)+28.06|0*57.41*cos(t)+1*57.41*sin(t)+-47.5}
\psline(-15.82,-10.48)(72.34,-10.96)
\psline[linewidth=2pt](-15.82,-10.48)(13.32,7.98)
\psline[linewidth=2pt](42.59,8.04)(72.34,-10.96)
\psline[linewidth=2pt](30.6,-10.73)(42.59,8.04)
\psline[linewidth=2pt](-15.82,-10.48)(25.16,-10.7)
\psline[linewidth=2pt](13.32,7.98)(25.16,-10.7)
\psline[linewidth=2pt](30.6,-10.73)(72.34,-10.96)
\psline(11.74,6.98)(12.78,5.64)
\psline(12.78,5.64)(14.17,6.63)
\psline(44.17,6.76)(42.94,5.42)
\psline(42.94,5.42)(41.81,6.39)
\parametricplot[linewidth=2.8pt]{0.6900067105221149}{1.3149902001887745}{1*57.41*cos(t)+0*57.41*sin(t)+28.06|0*57.41*cos(t)+1*57.41*sin(t)+-47.5}
\parametricplot[linewidth=2.8pt]{1.8305058922313753}{2.4407382457053424}{1*57.41*cos(t)+0*57.41*sin(t)+28.06|0*57.41*cos(t)+1*57.41*sin(t)+-47.5}
\psline(-1.09,20)(-0.27,2.69)
\psline(-15,20)(-5,-2)
\psline(53.32,20)(55.91,2.88)
\psline(14.19,20)(10.98,-1.9)
\psline(26.95,20)(14.19,-10.35)
\psline(44.01,20)(32.35,3.5)
\psline[linewidth=2pt](-15.82,-10.48)(-15.82,-20)\psline[linewidth=2pt,linestyle=dashed](-15.82,-30)(-15.82,-20)
\psline[linewidth=2pt](25,-10.48)(25,-20)\psline[linewidth=2pt,linestyle=dashed](25,-30)(25,-20)
\psline[linewidth=2pt](30.5,-10.48)(30.5,-20)\psline[linewidth=2pt,linestyle=dashed](30.5,-30)(30.5,-20)
\psline[linewidth=2pt](72.2,-10.48)(72.2,-20)\psline[linewidth=2pt,linestyle=dashed](72.2,-30)(72.2,-20)
\begin{scriptsize}
%\psdots[dotstyle=*,linecolor=blue](-1.09,19.25)
\rput[bl](-3,21){{$K_{2k-1}^{N+1}$}}
\rput[bl](-18,21){{$D_{2k-1}^{N+1}$}}
%\psdots[dotstyle=*,linecolor=blue](53.32,16.75)
\rput[bl](52,20){$K_{2k}^{N+1}$}
%\psdots[dotstyle=*,linecolor=blue](14.19,20.02)
\rput[bl](13,20){{$T_{2k-1}^{N+1}$}}
%\psdots[dotstyle=*,linecolor=blue](26.95,18.22)
\rput[bl](26,20){$H_{2k-1}^{N+1}$}
%\psdots[dotstyle=*,linecolor=blue](44.01,35.91)
\rput[bl](43,20){$D_k^{N}$}
\end{scriptsize}
\end{pspicture*}
\caption{Induction. The bold lines correspond to the new iteration}\label{HereditePsi}\end{center}
\end{figure}
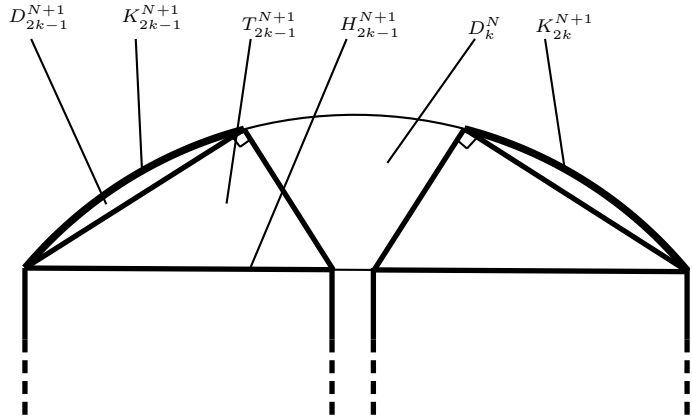

\begin{remark}\label{RmDefVinfty}It is easy to check that for $N\geq 0$:
\begin{enumerate}
\item  $T_{N+1}\subset D_N^+$,
\item\label{Arg1DefVinfty} if  $x\in \,\stackrel{\circ}{T_N}$ then $x\notin T_{N'}$ for $N'\geq N+1$ [here $T_0=\emptyset$].
\end{enumerate}
\end{remark}
 We now define 
\[
\begin{array}{cccc}
\tilde\Psi_{N+1}:&T_{N+1}\cup D_{N+1}^-&\to&\cup_{k=1}^{2^{N+1}}\C_k^{N+1}\\&x&\mapsto&\begin{cases}\Pi_{\C_k^{N+1}}(x)&\text{if }x\in T_k^{N+1}\\\Pi_{\C_k^{N+1}}[\tilde{\Psi}_N(x)] &\text{if }x\in \tilde{\Psi}_{N}^{-1}(\cup_{k=1}^{2^{N+1}}H_k^{N+1})\end{cases}
\end{array}.
\]
In Appendix \ref{AppeSepTr} [Lemma \ref{PropMainGeomHyp} and Remark \ref{RemPropMainGeomHyp}], it is proved  that for $N\geq1$, the triangles $T_k^N$ for $k=1,...,2^N$ are mutually disjoint. recursively, we find that all the $\tilde\Psi_N$'s are well-defined.%Thus the map  $\tilde\Psi_{N+1}$ is well defined is $\tilde\Psi_N$ makes sense.

And, as in the Initialization Step, we get  $\Psi_{N+1}$ from $\tilde\Psi_{N+1}$ by projecting $\cup_{k=1}^{2^{N+1}}\C_k^{N+1}$ on $\p\O$:
\[
\begin{array}{cccc}
\Psi_{N+1}:&T_{N+1}\cup D_{N+1}^-\cup D_{N+1}^+&\to&\K_{N+1}\\&x&\mapsto&\begin{cases}\Pi_{\p\O}[\tilde\Psi_{N+1}(x)]&\text{if }x\in T_{N+1}\cup D_{N+1}^-\\\Pi_{\p\O}(x) &\text{if }x\in D_{N+1}^+\end{cases}
\end{array}.
\]
It is easy to see that $\Psi_{N+1}(T_{N+1}\cup D_{N+1}^-\cup D_{N+1}^+)=\K_{N+1}$.
\section{Basic properties of $B_\infty$ and $\Psi_N$}\label{SectionSommePromqljkdsgn}
\subsection{Basic properties of $B_\infty$} 
We set $B_N={T_{N}\cup D_{N}^+\cup D_{N}^-}$. It is easy to check that for $N\geq0$ we have $B_{N+1}\subset B_N$ and  $\K\subset\p B_N$. Therefore we may define
\[
B_\infty=\cap_{N\geq0}\overline{B_N}
\]
which is compact and  satisfies $\K\subset\p B_\infty$.

We are going to prove:
\begin{lem}\label{LemADem}
The interior of $B_\infty$ is empty.
\end{lem}

%For $N\geq 1$, we  define $B_N={T_{N}\cup D_{N}^+\cup D_{N}^-}$. 

\begin{proof}[Proof of Lemma \ref{LemADem}]

From Lemma \ref{PropMainGeomHyp} [and Remark \ref{RemPropMainGeomHyp}] in Appendix \ref{AppeSepTr} combined with Hypothesis \eqref{MainGeomHyp}, we get two fundamental facts:
\begin{enumerate}
\item The triangles $T_1^N$, ... ,$T_{2^{N+1}}^N$ are mutually disjoint.
\item We have: 
\begin{equation}\label{R-Ouf1ii}
\H^1(H_k^{N+1})<\dfrac{\H^1(\mathcal{F}(\C_k^N))}{2}.
\end{equation}
\end{enumerate}

For a non empty set $A\subset\R^2$ we let
\[
{\rm rad}(A)=\sup\{r\geq0\,;\,\exists\,x\in A \text{ s.t. } \overline{B(x,r)}\subset A\}.
\]
Note that the topological interior of $A$ is empty if and only if ${\rm rad}(A)=0$.

On the one hand, it is not difficult to check that for sufficiently large $N$
\begin{equation}\label{StepNRadIEst}
{\rm rad}(B_N)={\rm rad}(B_N\cap D_N^-).
\end{equation}
On the other hand, using \eqref{R-Ouf1ii} we obtain for $N\geq1$: %and the  Incidence Theorem applied in a recursive way, we get that for $N\geq1$:
\begin{equation}\label{StepNRadIIEst}
{\rm rad}(B_{N+1}\cap D_{N+1}^-)\leq\dfrac{{\rm rad}(B_N\cap D_N^-)}{2}.
\end{equation}
Consequently, by combining \eqref{StepNRadIEst} and \eqref{StepNRadIIEst} we get the existence of $C_0$ s.t.
\begin{equation}\label{StepNRadIIIEst}
{\rm rad}(B_N)\leq\frac{C_0}{2^N}.
\end{equation}

Since $B_\infty=\cap_{N\geq0}B_N$, from  \eqref{StepNRadIIIEst} we get that ${\rm rad}(B_\infty)=0$. %The last estimate implies the expected result.

\end{proof}
\subsection{Basic properties of $\Psi_N$} 
We now prove the key estimate for $\Psi_N$:
\begin{lem}\label{LemInfoPsi}
There exists $b_N=o_N(1)$ s.t. for $N\geq1$ and $U$ a connected component of $B_N$, the restriction of $\Psi_N$ to $U$ is $(1+b_N)$-Lipschitz.
\end{lem}
\begin{proof}
Let $N\geq1$ and $U$ be a connected component of $B_N$. The restriction of  $\tilde{\Psi}_N$ to $U\cap(T_N\cup D_N^-)$ is obtained as composition of orthogonal projections on straight lines and thus is $1$-Lipschitz. 

There exists $b_N=o_N(1)$ s.t. the projection $P_N:=\Pi_{\p\O}$ defined in $\overline{D_{N}^+}$ is $(1+b_N)$-Lipschitz.  The functions ${\Psi_N}$ are either the composition of  $\tilde{\Psi}_N$ with $P_N$ or $\Psi_N=P_N$. Consequently the restriction of $\Psi_N$ to $U$ is $(1+b_N)$-Lipschitz.
\end{proof}

In the following we will not use $\Psi_N$ but  "its projection" on $\R$. %For $N\geq1$, recall that $\K_N=\cup_{k=1}^{2^N}K_k^N$ [see Section \ref{SectHerdityCOnstruKantor}]. 
For $N\geq1$ and $k\in\{1,...,2^N\}$, we let $B_k^N:=\Psi_N^{-1}(K_k^N)$ and we define
\[
\begin{array}{cccc}
\Proj_{k,N}:&B_k^N&\to&\R\\&x&\mapsto&\H^1(\curve{A\Psi_N(x)})
\end{array}
\]
where $\curve{A\Psi_N(x)}\,\subset\,\curve{AB}$ is defined by \eqref{DefCurve} as the smallest connected component of $\p\O\setminus\{A,\Psi_N(x)\}$ if $\Psi_N(x)\neq A$ and $\curve{A\Psi_N(x)}=\{A\}$ otherwise.
\begin{lem}\label{LPiLiplkhu}
For $N\geq1$ there exists $c_N\in(0,1)$ with $c_N=o_N(1)$ s.t.   for $k\in\{1,...,2^N\}$ the function $\Proj_{k,N}:B_k^N\to\R$ is $(1+c_N)$-Lipschitz.
\end{lem}
\begin{proof}
Let $N\geq1$, $k\in\{1,...,2^N\}$ and let $x,y\in B_k^N$ be s.t. $\Psi_N(x)\neq\Psi_N(y)$. It is clear that we have
\[
|\Proj_{k,N}(x)-\Proj_{k,N}(y)|=\H^1(\curve{\Psi_N(y)\Psi_N(x)})
\]
where $\curve{\Psi_N(y)\Psi_N(x)}\subset K_k^N$ is defined by \eqref{DefCurve} as the smallest connected component of $\p\O\setminus\{\Psi_N(y),\Psi_N(x)\}$.

Moreover, from Lemma \ref{LOrth} in Appendix \ref{SAppeLOrth}, we have the existence of $C\geq1$ independent of $N$ and $k$ s.t. for $x,y\in B_k^N$ s.t. $\Psi_N(x)\neq\Psi_N(y)$ we have [denoting $X:=\Psi_N(x),Y:=\Psi_N(y)$]
\[
\dist\left(X,Y\right)\leq\H^1\left(\curve{XY}\right)\leq\dist\left(X,Y\right)\left[1+C\dist\left(X,Y\right)\right]
\]
and
\[
\H^1(K_k^N)\leq\H^1(\C_k^N)\left[1+C\H^1(\C_k^N)\right].
\]
From Step 1 in the proof of Proposition \ref{defKinffty} [Appendix \ref{ProofOfPropDefKinkdfg}] we have
\[
\displaystyle\max_{k=1,...,2^N}\H^1(\C_k^N)\leq\left(\dfrac{2}{3}\right)^N.
\]
Thus letting $a_N:=\left(\dfrac{2}{3}\right)^N\left[1+C\left(\dfrac{2}{3}\right)^N\right]$ we have $a_N\to0$ and since $\curve{XY}\subset K_k^N$ we get:
\[
\dist\left(X,Y\right)\leq\H^1\left(\curve{XY}\right)\leq\H^1(K_k^N)\leq\H^1(\C_k^N)\left[1+C\H^1(\C_k^N)\right]\leq a_N(1+Ca_N).
\]
%Note that, from \eqref{StepNRadIIIEst}, we have the existence of $C_0\geq1$ s.t.
%\[
%\dist\left(X,Y\right)\leq{\rm rad}(B^{(k)}_N)\leq{\rm rad}(B_N)\leq\frac{C_0}{2^N}.
%\]
Thus, letting $\tilde{a}_N=\max\left\{a_N(1+Ca_N),|b_N|\right\}$ where $b_N$ is defined in Lemma \ref{LemInfoPsi}, we get% [with the help of Lemma \ref{LemInfoPsi}]
\begin{eqnarray*}
\H^1\left(\curve{XY}\right)=|\Proj_{k,N}(x)-\Proj_{k,N}(y)|&\leq&\H^1\left([{\Psi_N(y)\Psi_N(x)}]\right)\left(1+C\tilde{a}_N\right)
\\&\leq&(1+\tilde{a}_N)\left(1+C\tilde{a}_N\right)|x-y|.
\end{eqnarray*}
Therefore, letting $c_N$ be s.t. $1+c_N=(1+\tilde{a}_N)\left(1+C\tilde{a}_N\right)$ we have $c_N=o_N(1)$, $c_N$ is independent of $k\in\{1,...,2^N\}$ and $\Proj_{k,N}$ is $(1+c_N)$-Lipschitz.
\end{proof}

\section{Proof of Theorem \ref{MainTHM}}\label{SectionProofMainTHM}
We are now in position to prove  Theorem \ref{MainTHM}. %It is clear that we may  assume that $\O$ is a connected set.
This is done by contradiction. We assume that there exists a map $u_0\in BV(\O)$ which minimizes \eqref{LeProb}. 
\subsection{Upper bound}
The first step in the proof is the estimate
\begin{equation}\label{MainEstimate}
\int_\O|D u_0|\leq\|\1_\K\|_{L^1(\p\O)}=\H^1(\K).
\end{equation}
This estimate  is obtained by proving that for all $\v>0$ there exists $u_\v\in W^{1,1}(\O)$ s.t. $\tr_{\p\O}u_\v=\1_\K$ and 
\begin{equation}\label{MainEstimateEpsVersion}
\|\n u_\v\|_{L^1(\O)}\leq(1+\v)\|\tr_{\p\O}u_\v\|_{L^1(\O)}=(1+\v)\H^1(\K).
\end{equation}
%The last estimate is obtained by  Theorem 2.16 $\&$ Remark 2.17 in  \cite{giusti1984minimal}. The proof of this result is done  for special domains $\O$. 
%We sketch its adaptation for a $C^2$ 
Proposition \ref{ExtGiustiUpBound} in Appendix \ref{AppendixGiusti} gives the existence of such $u_\v$'s.

Clearly \eqref{MainEstimateEpsVersion} implies \eqref{MainEstimate}.

\subsection{Optimality of the upper bound}
In order to have a contradiction we follow the strategy of Spradlin and Tamasan in \cite{ST1}. We fix a sequence $(u_n)_n\subset C^1(\O)$ s.t.
\begin{equation}\label{NumX}
u_n\in W^{1,1}(\O)\,;\,u_n\to u\text{ in }L^1(\O)\,;\,\int_\O|\n u_n|\to\int_\O|D u_0|\,;\,\tr_{\p\O}u_n=\tr_{\p\O}u_0.
\end{equation}
Note that \eqref{NumX} implies 
\begin{equation}\label{NumXX}
\int_{F}|\n u_n|\to\int_F|Du|\text{ for all $F\subset\O$ relatively closed set.}
\end{equation}
Such a sequence can be obtained {\it via} partition of unity and smoothing ; see the proof of Theorem 1.17 in \cite{giusti1984minimal}. For the convenience of the reader a proof is presented in Appendix \ref{SLemApproxTechskdjfgh} [see Lemma \ref{LemApproxTechskdjfgh}].

For further use, let us note that the sequence $(u_n)_n$ constructed in Appendix \ref{SLemApproxTechskdjfgh} satisfies the following additional property:
\begin{equation}\nonumber%\label{ConsequenceFriedMol}
\left|\begin{array}{c}\text{If $u_0=0$ outside a compact set $L\subset\overline{\O}$ and if $\o$ is an open set }\\\text{s.t. $\dist(\o,L)>0$ then, for large $n$, $u_n=0$ in $\o$}\end{array}\right. .
\end{equation}
\begin{comment}
It satisfies $u_n\stackrel{BV}{\to} u_0$ and $\tr_{\p\O} u_n=\tr_{\p\O}u_O$. 

Here $u_n\stackrel{BV}{\to} u_0$ means $u_n\stackrel{L^1}{\to} u_0$ and $\int_\O|\n u_n|\to\int_\O|Du_0|$.

%On peut de plus supposer que $0\leq u\leq 1$ p.p.. En effet, dans le cas contraire, en prenant $\tilde{u}_n=\begin{cases}1&\text{dans }\{x\in\O\,;\,u(x)>1\}\\u&\text{dans }\{x\in\O\,;\,0\leq u(x)\leq1\}\\0&\text{dans }\{x\in\O\,;\,u(x)<0\}\end{cases}$ on a $\tilde{u}$ qui est une fonction de $BV(\O)$ et qui vérifie $\int_\O|D u|\leq\int_\O|D u|$ et $\tr_{\p\O}u=\1_\K$. 

%Par compacité, quitte à extraire, on a l'existence de $\tilde{u}\in BV(\O)$ tq $\tilde{u}_k\stackrel{L^1}{\to}\tilde{u}$, $\int_\O|\n u_k|\int_\O|\n \tilde u_k|\leq\int_\O|\n u_k|$ et $\tr_{\p\O}u_k=\1_\K$% avec égalité si et seulement si $0\leq u \leq1$ p.p..
%Puisque le Cantor n'est {\bf pas trop grand}, on peut supposer qu'au voisinage de $C$, $\p\O$ est paramètre par $\gamma:]-\eta,\eta[\to$
\end{comment}

For $x\in B_0$ we let 
\begin{equation}\label{DefVect0}
V_{0}(x)=\begin{cases}\nu_{\Pi_{\p\O}(x)}&\text{if }x\in D_{0}^+\\(0,1)&\text{if }x\in D_{0}^-\end{cases},
\end{equation}
and for $N\geq0$, $x\in B_{N+1}$ we let
\begin{equation}\label{DefVect}
V_{N+1}(x)=\begin{cases}V_N(x)&\text{if }x\in B_{N}\setminus \stackrel{\circ}{T}\phantom{a}\!\!\!\! ^{N+1}\\\nu_{\C_k^{N+1}}&\text{if }x\in\,\stackrel{\circ}{T}\phantom{a}\!\!\!\!_k ^{N+1}\end{cases},
\end{equation}
where, for $\sigma\in\p\O$, $\nu_{\sigma}$ is the normal outward of $\O$ in $\sigma$ and $\nu_{\C_k^{N+1}}$ is defined in Remark \ref{ConstructionFact}.\ref{RemCorde}.

We now prove the following lemma.
\begin{lem}\label{LemmeDefVinfty} When $N\to\infty$ we may define $V_\infty(x)$ a.e. $x\in B_\infty$ by
\begin{equation}\label{DefVectInfty}
\begin{array}{cccc}
V_\infty:&B_\infty&\to&\R^2\\&x&\mapsto&\lim_{N\to\infty}V_N(x)%&\text{ if }x\in B_\infty\\0&\text{ if }x\notin B_\infty\end{cases}
\end{array}.
\end{equation}
Moreover, from dominated convergence, we have:
\begin{equation}\nonumber%\label{ConvVn}
V_N\1_{B_N}\to V_\infty\1_{B_\infty}\text{ in $L^1(\O)$}.
\end{equation}

\end{lem}
\begin{proof}
If  $x\in B_\infty\setminus\cup_{N\geq1}{T_N}$, then we have $V_N(x)=V_0(x)$ for all $N\geq1$. Thus $\lim_{N\to\infty}V_N(x)=V_0(x)$.

For a.e. $x\in B_\infty\cap\cup_{N\geq1}{T_N}$ there exists $N_0\geq1$ s.t. $x\in \stackrel{\!\!\!\!\circ}{T_{N_0}}$. 
\begin{comment} Then $x\in B_N$ for all $N\geq1$. In order to prove that $V_\infty(x)$ is well defined it suffices to note that, from Remark \ref{RmDefVinfty}.\ref{Arg1DefVinfty}, if $x\in \stackrel{\!\!\!\!\circ}{T_{N_0}}$ for some $N_0\geq1$, then $x\notin T_{N}$ for $N>N_0$. Consequently, there is a dichotomy: 
\begin{itemize}
\item $x\in B_\infty\setminus\cup_{N\geq1}\stackrel{\!\!\circ}{T_N}$,
\item $\exists!\,N_0\geq1$ s.t. $x\in \stackrel{\!\!\!\!\circ}{T_{N_0}}$.
\end{itemize}
If  $x\in B_\infty\setminus\cup_{N\geq1}\stackrel{\!\!\circ}{T_N}$, then we have $V_N(x)=V_0(x)$ for all $N\geq1$. Thus $\lim_{N\to\infty}V_N(x)=V_0(x)$. Otherwise there exists a unique $N_0\geq1$ s.t. $x \in \,\stackrel{\!\!\!\!\circ}{T_{N_0}}$. \end{comment}Therefore for all $N>N_0$ we have $V_{N}(x)=V_{N_0}(x)$. Consequently $\lim_{N\to\infty}V_{N}(x)=V_{N_0}(x)$.
\end{proof}

This section is devoted to the proof of the following lemma:

\begin{lem}\label{Fondlemlower}
For all $w\in C^\infty\cap W^{1,1}(\O)$ s.t. $\tr_{\p\O}w=\1_\K$ we have
\[
\int_{B_\infty\cap\O}|\n w\cdot V_\infty|\geq\H^1(\K)%\dfrac{|\K|}{|\Psi_N|_{\rm Lip(B_N)}}
\]
where $V_\infty$ is the vector field defined in \eqref{DefVectInfty}.
%avec $|\Psi_N|_{\rm Lip(B_N)}=1+o_N(1)$ qui est la semi-norme Lipschitz de $\Psi_N$ et $v_N$ est le champs vectoriel défini dans \eqref{DefVect}.
\end{lem}
\begin{remark}
Since $|V_\infty(x)|=1$ for a.e. $x\in B_\infty$, it is clear that Lemma \ref{Fondlemlower} implies that for all $n$  we have
\[
\int_{B_\infty\cap\O}|\n u_n|\geq\H^1(\K).
\]
From \eqref{NumXX} we have:
\[
\int_{B_\infty\cap\O}|D u_0|\geq\H^1(\K).
\]
Section \ref{SectTransversArg} is devoted to a sharper argument than above to get
\[
\int_{B_\infty\cap\O}|\n u_n|\geq \int_{B_\infty\cap\O}|\n u_n\cdot V_\infty|+\delta
\] with $\delta>0$ is independent of $n$. The last estimate will imply  $\int_{B_\infty\cap\O}|D u_0|\geq\H^1(\K)+\delta$ which will be the contradiction we are looking for.
\end{remark}
\begin{proof}[Proof of Lemma \ref{Fondlemlower}]
We will first prove that for $w\in C^\infty\cap W^{1,1}(\O)$ s.t. $\tr_{\p\O}w=\1_\K$ we have
\begin{equation}\label{AnnexeFondlms}
\int_{B_N\cap\O}|\n w\cdot V_N|\geq\dfrac{\H^1(\K)}{1+o_N(1)}.
\end{equation}
where  
%M\[
%|\Psi_N|_{\rm Lip(B_N)}=\sup\left\{\frac{|\Psi_N(x)-\Psi_N(x)|}{|x-y|}\,;\,x,y\in B_N,\,x\neq y\right\},
%\] is the Lipischitz semi-norm of $\Psi_N$ and 
$V_N$ is the vector field defined in \eqref{DefVect0} and \eqref{DefVect}.

Granted \eqref{AnnexeFondlms}, we conclude as follows: if $w\in C^\infty\cap W^{1,1}(\O)$ s.t. $\tr_{\p\O}w=\1_\K$, then
\begin{eqnarray*}
\int_{B_\infty\cap\O}|\n w\cdot V_\infty|&=&\lim_{N\to\infty}\int_{B_N\cap\O}|\n w\cdot V_N|
\\&\geq&\lim_{N\to\infty}\dfrac{\H^1(\K)}{1+o_N(1)}=\H^1(\K),
\end{eqnarray*}
by dominated convergence.

%Thus, with the help of Lemma \ref{LemInfoPsi} [$|\Psi_N|_{\rm Lip(B_N)}=1+o_N(1)$] we have
%\[
%\int_{B_\infty\cap\O}|\n w\cdot V_\infty|\geq|\K|.%\dfrac{|\K|}{|\Psi_N|_{\rm Lip(B_N)}}
%\]

It remains to prove \eqref{AnnexeFondlms}.  We fix $w\in C^\infty\cap W^{1,1}(\O)$ s.t. $\tr_{\p\O}w=\1_\K$.  Using the Coarea Formula we have for $N\geq1$ and $k\in\{1,...,2^N\}$, with the help of Lemma \ref{LPiLiplkhu}, we have
\begin{eqnarray*}
(1+c_N)\int_{B^{(k)}_N\cap\O}|\n w\cdot V_N|&\geq&\int_{B^{(k)}_N\cap\O}|\n\Proj_{k,N}||\n w\cdot V_N|\\
&=&\int_{\R}{\rm d}t\int_{\Proj_{k,N}^{-1}(\{t\})\cap\O}|\n w\cdot V_N|.
\end{eqnarray*}
Here, if  $\Proj_{k,N}^{-1}(\{t\})$ is non trivial, then  $\Proj_{k,N}^{-1}(\{t\})$ is a polygonal line: 
\[
\Proj_{k,N}^{-1}(\{t\})=I_{ \sigma(t,k,N)}\cup I_{k,N,t}^1\cup\cdots\cup I_{k,N,t}^{N+1}
\]
where 
\begin{itemize}
\item ${ \sigma(t,k,N)}\in[AB]$ is s.t. $[AB]\cap\Proj_{k,N}^{-1}(\{t\})=\{ \sigma(t,k,N)\}$,
\item  $I_{ \sigma(t,k,N)}$ is defined in \eqref{defISigma}, 
\item for $l=1,...,N$ we have $I_{k,N,t}^l=\Proj_{k,N}^{-1}(\{t\})\cap T_{N+1-l}$,
\item $I_{k,N,t}^{N+1}=\Proj_{k,N}^{-1}(\{t\})\cap D_N^+$.
\end{itemize}
From the Fundamental Theorem of  calculus and from the definition of $V_N$, denoting  
\begin{itemize}
\item $I_{ \sigma(t,k,N)}=[M_0,M_1]$ [where $M_0\in\p\O\setminus\curve{AB}$ and $M_1={ \sigma}(t,k,N)$], 
\item $I_{k,N,t}^l=[M_l,M_{l+1}]$, $l=1,...,N+1$ and $M_{N+2}\in K_k^N$, 
\end{itemize}
we have for a.e. $t\in\Proj_{k,N}(K_k^N)$ and using the previous notation, 
\[
\int_{[M_l,M_{l+1}]}|\n w\cdot V_N|\geq|w(M_{l+1})-w(M_l)|.
\]
Here we used the convention  $w(M_l)=\tr_{\p\O} w(M_l)$ for $l=0\&N+2$.

Therefore for a.e $t\in\Proj_{k,N}(K_k^N)$ we have
\[
\int_{\Proj_{k,N}^{-1}(\{t\})\cap\O}|\n w\cdot V_N|\geq|\tr_{\p\O} w(M_{N+2})-\tr_{\p\O} w(M_0)|=\1_\K(M_{N+2}).
\] 
Since $\K\subset \K_N=\cup_{k=1}^{2^N}K_k^N$, we may thus deduce that
\begin{eqnarray*}
(1+c_N)\int_{B_N\cap\O}|\n w\cdot V_N|=(1+c_N)\sum_{k=1}^{2^N}\int_{B^{(k)}_N\cap\O}|\n w\cdot V_N|&\geq&\int_{\curve{AB}}\1_\K=\H^1(\K).
\end{eqnarray*}
The last estimate clearly implies  \eqref{AnnexeFondlms} and completes the proof of Lemma \ref{Fondlemlower}.
\end{proof}
\subsection{Transverse argument}\label{SectTransversArg}
We assumed that there exists a map $u_0$ which solves Problem \eqref{LeProb}. 

We investigate the following dichotomy: 
\begin{itemize}
\item $u_0\not\equiv0$ in $\O\setminus B_\infty$;
\item $u_0\equiv0$ in $\O\setminus B_\infty$.
\end{itemize}
We are going to prove that both cases lead to a contradiction.

\subsubsection{The case $u_0\not\equiv0$ in $\O\setminus B_\infty$ }

We thus have $\displaystyle\int_{\O\setminus B_\infty}|u_0|>0$. In this case, since $(\tr_{\p\O} u_0)_{|\p\O\setminus\p B_\infty}\equiv0$, we have 
\begin{equation}\label{ConsWeakPoincLemma}
\delta:=\displaystyle\int_{\O\setminus B_\infty}|D u_0|>0.
\end{equation}
Estimate \eqref{ConsWeakPoincLemma} is a direct consequence of the following lemma applied on each connected components of $\O\setminus B_\infty$. 
\begin{lem}\label{LWeakPoincEst}[Weak Poincaré lemma]
Let $\o\subset\R^2$ be an open connected set. Assume that there exist $x_0\in\p\o$ and $r>0$ s.t. $\o\cap B(x_0,r)$ is Lipschitz.

If $u\in BV(\o)$ satisfies  $\tr_{\p\o\cap B(x_0,r)}=0$ and $\int_\o|Du|=0$ then $u=0$.
\end{lem}
Lemma \ref{LWeakPoincEst} is proved in Appendix \ref{ProofWeakPoincLem}.

Recall that we fixed  a sequence  $(u_n)_n\subset C^1\cap W^{1,1}(\O)$ satisfying \eqref{NumX}.%s.t. $u_n\tos u_0$ and $\tr_{\p\O} u_n=\tr_{\p\O}u_0$ [Lemma \ref{LemApproxTechskdjfgh} in Appendix \ref{SLemApproxTechskdjfgh}]. 

In particular, for sufficiently large $n$, we have
\[
\int_{\O\setminus B_\infty}|\n u_n|>\dfrac{\delta}{2}.
\]
Thus, from Lemma \ref{Fondlemlower} and the fact that $|V_\infty(x)|=1$ for a.e. $x\in B_\infty$,
\[
\int_{\O}|\n u_n|\geq\int_{B_\infty}|\n u_n\cdot V_\infty|+\int_{\O\setminus B_\infty}|\n u_n|\geq\H^1(\K)+\frac{\delta}{2}.
\]
This implies
\[
\int_{\O}|D u_0|=\lim_n \int_{\O}|\n u_n|\geq\H^1(\K)+\frac{\delta}{2}
\]
which is in contradiction with \eqref{MainEstimate}.
\subsubsection{The case  $u_0\equiv0$  in $\O\setminus B_\infty$}

We first note that, since $\tr_{\p D_0^+}u_0\not\equiv0$, there exists a triangle  $T^{N_0}_k$ s.t. $\int_{T_k^{N_0}}|u_0|>0$. We fix such a triangle  $T^{N_0}_k$ and we let $\alpha$ be the vertex corresponding to the right angle.

We let  $\tilde\repere=(\alpha,\tilde {\bf e}_1,\tilde {\bf e}_2)$ be the direct orthonormal frame centered in  $\alpha$ where  $\tilde {\bf e}_2=\nu_{\C^{N_0}_k}$ [$\nu_{\C^{N_0}_k}$ is defined Remark \ref{ConstructionFact}.\ref{RemCorde}],{\it i.e.}, the directions of the new frame are given by the side of the right-angle of $T_k^{N_0}$.

It is clear that for $N\geq{N_0}$ we have ${V_N}\equiv\tilde {\bf e}_2$ in $\stackrel{\!\!\!\!\!\circ}{T_k^{N_0}}$.

By construction of  $B_\infty$,  $T_k^{N_0}\cap B_\infty$ is a union of segments parallel to ${\tilde e}_2$, {\it i.e.} ${\1_{B_\infty}}_{| T_k^{N_0}}(s,t)$ depends only on the first variable "$s$" in the frame $\tilde\repere$.

Since $\int_{T_k^{N_0}}|u_0|>0$, in the frame $\tilde\repere$, we may find $a,b,c,d\in\R$ s.t.,  considering the rectangle (whose sides are parallel to the direction of $\tilde\repere$)
\[
\Pab:=\left\{\alpha+s\tilde {\bf e}_1+t\tilde {\bf e}_2\,;\,(s,t)\in[a,b]\times[c,d]\right\}\subset T_k^{N_0}
\]  we have
\[
\int_{\Pab}|u_0|>0.
\]
Since from Lemma \ref{LemADem} the set  $B_\infty$ has an empty interior [and that ${\1_{B_\infty}}_{| T_k^{N_0}}(s,t)$ depends only on the first variable in the frame $\tilde\repere$], we may find $a'<b'$ s.t. 
\begin{itemize}
\item$[a',b']\times[c,d]\subset[a,b]\times[c,d]$,
\item $\Sab\cap B_\infty=\emptyset$ with $\Sab:=\left\{\alpha+s\tilde {\bf e}_1+t\tilde {\bf e}_2\,;\,(s,t)\in\{a',b'\}\times[c,d]\right\}$
\item $\displaystyle\delta:=\int_{\Pab'}|u_0|>0$ with $\displaystyle\Pab':=\left\{\alpha+s\tilde {\bf e}_1+t\tilde {\bf e}_2\,;\,(s,t)\in[a',b']\times[c,d]\right\}$.
\end{itemize}
Moreover, since $\Sab$ and $B_\infty$ are compact sets with empty intersection, we may find $\mathcal{V}$, an open neighborhood of  $\Sab$ s.t. $\dist(\mathcal{V},B_\infty)>0$.\\

Noting that $u_0\equiv0$ in $\O\setminus B_\infty$, from Lemma \ref{LemApproxTechskdjfgh} [in Appendix \ref{SLemApproxTechskdjfgh}] % and that the sequence $(u_n)_n$ is obtained by the convolution of  $u_0$ with a $C^\infty$-mollifier with compact support, 
it follows that for sufficiently large $n$ we have 
\begin{itemize}
\item $u_n\equiv0$ in $\Sab$,
\item  $\displaystyle\int_{\Pab'}|u_n|>\dfrac{\delta}{2}$.
\end{itemize}

Consequently, from a standard Poincaré inequality
\[
\int_{\Pab'}|\p_{{\tilde {\bf e}}_1} u_n|\geq\dfrac{2}{b'-a'}\int_{\Pab'}|u_n|>\dfrac{{\delta}}{b'-a'}=:\delta'.
\]

Therefore  $\int_{\Pab'}|\p_{{\tilde {\bf e}}_1}  u_n|>\delta'$, $\int_{\Pab'}|\p_{{\tilde {\bf e}}_2}  u_n|\leq 2\H^1(\K)$ and then by Lemma  3.3  in \cite{ST1} we obtain:
\[
\int_{\Pab'}|\n u_n|\geq\int_{\Pab'}|\p_{{\tilde {\bf e}}_2} u_n|+\dfrac{\delta'^2}{4\H^1(\K)+\delta'}.
\]
Thus, from Lemma  \ref{Fondlemlower}, for sufficiently large  $n$:
\[
\int_\O|\n u_n|\geq \H^1(\K)+\dfrac{\delta'^2}{4\H^1(\K)+\delta'}-o_n(1).
\]
From the convergence in $BV$-norm of $u_n$ to $u_0$ we have 
\[
\int_\O|D u_0|\geq \H^1(\K)+\dfrac{\delta'^2}{4\H^1(\K)+\delta'}.
\]
Clearly this last assertion contradicts \eqref{MainEstimate} and ends the proof of Theorem \ref{MainTHM}.\\

\appendix
\section*{Appendices}

\section{A smoothing result}\label{SLemApproxTechskdjfgh}
We first state a standard approximation lemma for $BV$-functions.
\begin{lem}\label{LemApproxTechskdjfgh}
Let $\O\subset\R^2$ be a bounded Lipschitz open set and let $u\in BV(\O)$. There exists a sequence $(u_n)_n\subset C^1(\O)$ s.t.
\begin{enumerate}
\item $u_n\tos u$ in the sense that $u_n\to u$ in $L^1(\O)$ and $\displaystyle\int_\O|\n u_n|\to\int_\O|D u|$,
\item $\tr_{\p\O} u_n=\tr_{\p\O} u$ for all $n$,
\item for $k\in\{1,2\}$, 
\[
\int_\O|\p_k u_n|\to\int_\O|D_k u|:=\sup\left\{\int_{\O} u\p_k\xi\,;\,\xi\in C_c^1({\O},\R)\text{ and }|\xi|\leq1\right\}\]
\item If $u=0$ outside a compact set $L\subset\overline{\O}$ and if $\o$ is an open set s.t. $\dist(\o,L)>0$ then, for large $n$, $u_n=0$ in $\o$.
\end{enumerate}
\end{lem}
\begin{proof}
The first assertion is quite standard. It is for example proved in   \cite{AG78} [Theorem 1]. We present below the classical example of sequence for such approximation result [we follow the presentation of \cite{giusti1984minimal}, Theorem 1.17].

Let $\O\subset\R^2$ be a bounded Lipschitz open set and let $u\in BV(\O)$.

For $n\geq1$, we let $\v=1/n$. We may fix $m\in\N^*$ sufficiently large s.t. letting for $k\in\N$
\[
\O_k=\left\{x\in\O\,;\,\dist(x,\p\O)>\dfrac{1}{m+k}\right\}
\]
we have
\[
\int_{\O\setminus\O_0}|D u|<\v.
\]
We fix now $A_1:=\O_2$ and for $i\in\N\setminus\{0,1\}$ we let $A_i=\O_{i+1}\setminus\overline{\O_{i-1}}$. It is clear that $(A_i)_{i\geq1}$ is a covering of $\O$ and that each  point in $\O$ belongs to at most three of the sets $(A_i)_{i\geq1}$.

We let $(\varphi_i)_{i\geq1}$ be a partition of unity subordinate to the covering $(A_i)_{i\geq1}$, {\it i.e.}, $ \varphi_i\in C^\infty_c(A_i),\,0\leq\varphi_i\leq1$ and $\sum_{i\geq1}\varphi_i=1$ in $\O$.

We let $\eta\in C^\infty_c(\R^2)$ be s.t. ${\rm supp}(\eta)\subset B(0,1)$, $\eta\geq0$, $\int\eta=1$ and for $x\in\R^2$ $\eta(x)=\eta(|x|)$. For $t>0$ we let $\eta_t=t^{-2}\eta(\cdot/t)$.

As explained in \cite{giusti1984minimal}, for $i\geq1$, we may choose $\v_i\in(0,\v)$ sufficiently small s.t. 
\[
\begin{cases}
\supp(\eta_{\v_i}*(u\varphi_i))\subset A_i
\\
\displaystyle\int_\O|\eta_{\v_i}*(u\varphi_i)-u\varphi_i|<\dfrac{\v}{2^i}
\\
\displaystyle\int_\O|\eta_{\v_i}*(u\n\varphi_i)-u\n\varphi_i|<\dfrac{\v}{2^i}
\end{cases}.
\]
Here $*$ is the convolution operator. 

Define
\[
u_n:=\sum_{i\geq1}\eta_{\v_i}*(u\varphi_i).
\]
In some neighborhood of each point $x\in\O$ there are only finitely many nonzero terms in the sum defining $u_n$. Thus $u_n$ is well defined and smooth in $\O$. 

Moreover, we may easily check that 
\[
\|u_n-u\|_{L^1(\O)}+\left|\int_\O|Du|-\int_\O|\n u_n|\right|<\v\text{ [here $\v=1/n$]}.
\]
Thus the previous estimate proves that $(u_n)$ satisfies the first assertion, {\it i.e}, $u_n\tos u$.

As claimed in \cite{giusti1984minimal} [Remark 2.12] we have $\tr_{\p\O}u_n=\tr_{\p\O}u$ for all $n$. Thus the second assertion is satisfied.

We now prove the third assertion. %Let $k\in\{1,2\}$ and let $(\xi_l)_{l\geq1}\subset C_c^1({\O},\R)$ be s.t. $|\xi_l|\leq1$ and 
%\[
%\int_{\O} u\p_k\xi_l\underset{l\to\infty}{\to}\int_\O|D_k u|.
%\]
%For fixed $l\geq1$ we have from the $L^1$-convergence $\int_{\O} u_n\p_k\xi_l\underset{n\to\infty}{\to}\int_{\O} u\p_k\xi_l$. Thus there is $n_l\geq1$ s.t. for $n\geq n_l$
%\[
%\left|\int_{\O} u_n\p_k\xi_l-\int_{\O} u\p_k\xi_l\right|<\frac{1}{l}.
%\]
%Therefore, for $n\geq n_l$ we have [with an integration by parts and $|\xi_l|\leq1$]
%\[
%\int_{\O} u\p_k\xi_l\leq\int_{\O} u_n\p_k\xi_l+\frac{1}{l}\leq\int_{\O} |\p_ku_n|+\frac{1}{l}
%\]
%Consequently, letting $l\to\infty$ we get
Since $u_n\to u$ in $L^1(\O)$, by inferior semi continuity we easily get for  $k\in\{1,2\}$
\[
\int_{\O}|D_ku|\leq\liminf_{n\to\infty}\int_{\O} |\p_ku_n|.
\]
We now prove $\displaystyle\int_{\O}|D_ku|\geq\limsup_{n\to\infty}\int_{\O} |\p_ku_n|$.

Let $\xi\in C_c^1(\O,\R)$ with $|\xi|\leq1$. Since $\eta$ is a symmetric mollifier and $\sum \varphi_i=1$ we have

%Since $\p_k\left(\sum_{i\geq1}\varphi\right)=0$ we get in $\O$
%\[
%\sum_{i\geq1}\left\{\eta_{\v_i}*(u\varphi_i)\p_k\xi-u\varphi_i\p_k(\eta_{\v_i}*\xi)\right\}
%\]We have
\begin{eqnarray*}
\int_\O u_n\p_k\xi&=&\sum_{i\geq1}\int_\O\eta_{\v_i}*(u\varphi_i)\p_k\xi
\\&=&\sum_{i\geq1}\int_\O u\varphi_i\p_k(\eta_{\v_i}*\xi)
\\&=&\sum_{i\geq1}\int_\O u\p_k[\varphi_i(\eta_{\v_i}*\xi)]-\sum_{i\geq1}\int_\O u\p_k\varphi_i(\eta_{\v_i}*\xi)
\\&=&\sum_{i\geq1}\int_\O u\p_k[\varphi_i(\eta_{\v_i}*\xi)]-\sum_{i\geq1}\int_\O \xi\left[\eta_{\v_i}*(u\p_k\varphi_i)-u\p_k\varphi_i\right].
\end{eqnarray*}
On the one hand we have [note that $\varphi_i(\eta_{\v_i}*\xi)\in C^1_c(A_i)$ and $|\varphi_i(\eta_{\v_i}*\xi)|\leq1$]
\begin{eqnarray*}
\left|\sum_{i\geq1}\int_\O u\p_k[\varphi_i(\eta_{\v_i}*\xi)]\right|&=&\left|\int_{A_1} u\p_k[\varphi_i(\eta_{\v_i}*\xi)]+\sum_{i\geq2}\int_{A_i} u\p_k[\varphi_i(\eta_{\v_i}*\xi)]\right|
\\
&\leq&\int_{\O}|D_ku|+\sum_{i\geq2}\int_{A_i}|D_k u|
\\
&\leq&\int_{\O}|D_ku|+3\int_{\O\setminus\O_0}|D_k u|
\\
&\leq&\int_{\O}|D_ku|+3\v.
\end{eqnarray*}
Here we used that each  point in $\O$ belongs to at most three of the sets $(A_i)_{i\geq1}$, for $i\geq2$ we have $A_i\subset\O\setminus\O_0$ and 
\[
\int_{\O\setminus\O_0}|D_k u|\leq\int_{\O\setminus\O_0}|D u|<\v.
\]
On the other hand, since for $i\geq1$ $\displaystyle\int_\O|\eta_{\v_i}*(u\n\varphi_i)-u\n\varphi_i|<\dfrac{\v}{2^i}$, we get
\[
\left|\sum_{i\geq1}\int_\O \xi\left[\eta_{\v_i}*(u\p_k\varphi_i)-u\p_k\varphi_i\right]\right|\leq\sum_{i\geq1}\int_\O \left|\eta_{\v_i}*(u\p_k\varphi_i)-u\p_k\varphi_i\right|<\v.
\]
Consequently
\[
\sup\left\{\int_{\O} u_n\p_k\xi\,;\,\xi\in C_c^1({\O},\R)\text{ and }|\xi|\leq1\right\}=\int_{\O} |\p_ku_n|\leq\int_{\O} |D_ku|+4\v
\]
and thus $\displaystyle\limsup_n\int_{\O} |\p_ku_n|\leq\int_{\O} |D_ku|.$ This inequality in conjunction with $\displaystyle\liminf_n\int_{\O} |\p_ku_n|\geq\int_{\O} |D_ku|$ proves the third assertion of Lemma \ref{LemApproxTechskdjfgh}.

The last assertion of Lemma \ref{LemApproxTechskdjfgh} is a direct consequence of the definition of the $u_n$'s.
\end{proof}
\section{Proofs of Lemma \ref{TecLem1}, Lemma \ref{TecLem2}, Lemma \ref{TecLem3} and Lemma \ref{LWeakPoincEst}}\label{Proof3Lemma}
\subsection{Proof of Lemma \ref{TecLem1}}\label{AppProTecLem1}
Let $u\in BV(\Q)$.  We prove that
\[
\int_{\Q}|D_2 u|\geq\int_0^1|\tr_{\p\Q}  u(\cdot,0)-\tr_{\p\Q} u(\cdot,1)|.
\]
From Lemma \ref{LemApproxTechskdjfgh}, there exists $(u_n)_n\subset C^1(\Q)$ s.t. $\tr_{\p\Q} u_n=\tr_{\p\Q} u$, $u_n\tos u$ and 
\[
\int_{\Q}|\p_2 u_n|\to\int_{\Q}|D_2 u|.
\]
From Fubini's theorem and the Fundamental theorem of calculus we have
\begin{eqnarray*}
\int_{\Q}|\p_2 u_n|&=&\int_0^1{\rm d}x_1\int_0^1|\p_2 u_n(x_1,x_2)|{\rm d}x_2
\\
&\geq&\int_0^1{\rm d}x_1\left|\int_0^1\p_2 u_n(x_1,x_2){\rm d}x_2\right|
\\
&=&\int_0^1{\rm d}x_1\left| \tr_{\p\Q} u_n(x_1,1)-\tr_{\p\Q} u_n(x_1,0)\right|
\\
&=&\int_0^1\left| \tr_{\p\Q} u(\cdot,1)-\tr_{\p\Q} u(\cdot,0)\right|.
\end{eqnarray*}
Since $\int_{\Q}|\p_2 u_n|\to\int_{\Q}|D_2 u|$, Lemma \ref{TecLem1} is proved.
\subsection{Proof of Lemma \ref{TecLem2}}\label{AppProTecLem2} Let $\O$ be a planar open set. Let $u\in BV(\O)$ be s.t.
\[
\int_{\O}|D u|=\int_{\O}|D_2 u|.
\]
We prove that $\displaystyle\int_{\O}|D_1 u|=0$. We argue by contradiction and we assume that $\displaystyle\int_{\O}|D_1 u|>0$, {\it i.e.}, there exists $\xi\in C_c^1(\O)$ s.t. $|\xi|\leq1$ and 
\[
\eta:=\int_\O u\p_1 \xi>0.
\]
Let $(\xi_n)_n\subset C_c^1(\O)$ be s.t. $|\xi_n|\leq1$ and 
\[
\eta_n:=\int_\O u\p_2 \xi_n\to\int_{\O}|D_2 u|.
\]
%We may assume that $\displaystyle\eta_n<\int_{\O}|D_2 u|$.

For $(\alpha,\beta)\in\{x\in\R^2\,;\,|x|\leq1\}$ we let $\xi_{\alpha,\beta}^{(n)}=(\alpha\xi,\beta\xi_n)\in C^1_c(\O,\R^2)$. Clearly, $|\xi_{\alpha,\beta}^{(n)}|\leq1$ and 
\begin{equation}\label{numXXX}
\int_{\O}|D u|\geq\int_{\O}u\Div(\xi_{\alpha,\beta}^{(n)})=\alpha\eta+\beta\eta_n.%\to\alpha\eta+\beta\int_{\O}|D u|.
\end{equation}
If we maximize the right hand side of \eqref{numXXX} w.r.t. $(\alpha,\beta)\in\{x\in\R^2\,;\,|x|\leq1\}$, then we find with $(\alpha,\beta)=\left(\dfrac{\eta}{\sqrt{\eta^2+\eta_n^2}},\dfrac{\eta_n}{\sqrt{\eta^2+\eta_n^2}}\right)$
\[
\int_{\O}|D u|\geq\sqrt{\eta^2+\eta_n^2}\underset{n\to\infty}{\to}\sqrt{\eta^2+\left(\int_{\O}|D u|\right)^2}>\int_{\O}|D u|.
\]
\begin{comment}
We let $\displaystyle\alpha_n:=\frac{2}{\eta}\left(\int_{\O}|D_2 u|-\eta_n\right)$. Note that $\alpha_n\to0$ and thus for sufficiently large $n$ we have
\begin{eqnarray*}
\int_{\O}|D u|&\geq&\int_{\O}u\,\Div\left(\xi_{\alpha_n,\sqrt{1-\alpha_n^2}}^{(n)}\right)
\\
&=&\alpha_n\eta+\eta_n\sqrt{1-\alpha_n^2}
\\
&=&\eta_n+\alpha_n\eta+\mathcal{O}(\alpha_n^2)
\\
&=&\int_{\O}|D u|+\frac{\eta}{2}\alpha_n+\mathcal{O}(\alpha_n^2).
\end{eqnarray*}
Consequently, for sufficiently large $n$, we get $\int_{\O}|D u|>\int_{\O}|D u|$ which is impossible.
\end{comment}
This is a contradiction.

\subsection{Proof of Lemma \ref{TecLem3}}\label{AppProTecLem3} Let $u\in BV(\Q)$ satisfying $\tr_{\p\Q} u=0$ in $\{0\}\times[0,1]$. We are going to prove that
\[
\int_{\Q}| u|\leq \int_{\Q}|D_1 u|.
\]
Let $(u_n)_n\subset C^1(\O)$ be given by Lemma \ref{LemApproxTechskdjfgh}. Using the Fundamental theorem of calculus we have for  $(x_1,x_2)\in\Q$
\[
|u_n(x_1,x_2)|\leq\int_0^{x_1}|\p_1 u_n(t,x_2)|{\rm d}t\leq\int_0^{1}|\p_1 u_n(t,x_2)|{\rm d}t.
\]
Therefore, from Fubini's theorem, we get
\[
\int_{\Q}| u_n|\leq\int_\Q {\rm d}x_1{\rm d}x_2\int_0^{1}|\p_1 u_n(t,x_2)|{\rm d}t=\int_0^1 {\rm d}x_2\int_0^{1}|\p_1 u_n(t,x_2)|{\rm d}t=\int_Q|\p_1 u_n|.
\]
It suffices to see that $\int_{\Q}| u_n|\to\int_{\Q}| u|$ and $\int_Q|\p_1 u_n|\to\int_Q|D_1 u|$ to get the result.
\subsection{Proof of Lemma \ref{LWeakPoincEst}}\label{ProofWeakPoincLem}
Let $\o\subset\R^2$ be an open connected set. Assume there exist $x_0\in\p\o$ and $r>0$ s.t. $\o\cap B(x_0,r)$ is Lipschitz.

Let $u\in BV(\o)$ satisfying  $\tr_{\p\o\cap B(x_0,r)}u=0$ and $\int_\o|Du|=0$. We are going to prove that $u=0$. On the one hand, since $\int_\o|Du|=0$, we get $u=C$ with $C\in\R$ a constant. We thus have $\tr_{\p\o\cap B(x_0,r)}u=C$.  Consequently $C=0$ and $u\equiv 0$.  

\section{Results related to  the Cantor set $\K$}\label{AppPreuveFat}
\subsection{Justification of Remark \ref{ConstructionFact}.\eqref{RemCorde}}\label{ProofOfChianteRemark}

We prove the following lemma:
\begin{lem}\label{LemUniqnessWelde}
Let $\eta>0$ and let $f\in C^2([0,\eta],\R)$ be s.t. $\eta<\dfrac{1}{2\|f'\|_{L^\infty([0,\eta])}\|f''\|_{L^\infty([0,\eta])}}$. We denote $C_f$ the graph of $f$ in an orthonormal frame $\repereO$.%=( A,{\bf e}_1,{\bf e}_2)$.

For $0\leq a<b\leq\eta$, denoting $\C$ the chord $[(a,f(a)),(b,f(b))]$, for any  straight line $D$ orthogonal to $\C$ s.t. $D\cap\C\neq\emptyset$, the straight line $D$ intersect $C_{f,a,b}$ at exactly one points where $C_{f,a,b}$ is the part of $C_f$ delimited by $(a,f(a))$ and $(b,f(b))$.
\end{lem} 
\begin{remark}
We may state an analog result with $f\in C^1$ where we use the modulus of continuity of $f'$ instead of $\|f''\|_\infty$ in the hypothesis.
\end{remark}
\begin{proof}
The key point here is uniqueness. Indeed, for $0\leq a<b\leq\eta$ and $\C,D$ as in the lemma, we may easily prove that $C_{f,a,b}\cap D\neq\emptyset$ by solving an equation. [We do not use $\eta<(2\|f'\|_{L^\infty([0,\eta])}\|f''\|_{L^\infty([0,\eta])})^{-1}$ for the existence]

In contrast with the existence of an intersection point, its uniqueness is valid only for $\eta$ not too large. To prove uniqueness we argue by contradiction and we consider $f$ and $\eta$ as in lemma and we assume that there exist two points  $0\leq a<b\leq\eta$ s.t.  there exist $a\leq x<y\leq b$ s.t. the segments $[(x,f(x)),(y,f(y))]$ and $[(a,f(a)),(b,f(b))]$ are orthogonal. Note that with this hypothesis the straight line $D:=((x,f(x)),(y,f(y)))$ is orthogonal to the chord $\C:=[(a,f(a)),(b,f(b))]$.

So we get
\[
\frac{f(y)-f(x)}{y-x}=-\frac{b-a}{f(b)-f(a)}.
\]
From the Mean Value Theorem, there exist $c\in(x,y)$ and $\tilde{c}\in(a,b)$ s.t. $f'(c)=-\dfrac{1}{f'(\tilde{c})}$. Consequently
\begin{equation}\label{ContradictionPerpLame}
f'(c)\times[f'(\tilde{c})-f'(c)]=-1-[f'(c)]^2.
\end{equation}
From the hypothesis $\eta<(2\|f'\|_{L^\infty([0,\eta])}\|f''\|_{L^\infty([0,\eta])})^{-1}$, we have
\[
\left|f'(\tilde{c})-f'(c)\right|\leq \eta\|f''\|_{L^\infty([0,\eta])}< \dfrac{1}{2\|f'\|_{L^\infty([0,\eta])}}.
\]
Therefore,  we get 
\[
\left|f'(c)\times[f'(\tilde{c})-f'(c)]\right|< \dfrac{1}{2}
\]
which is in contradiction with \eqref{ContradictionPerpLame}.
\end{proof}
\subsection{Two preliminary results}\label{SAppeLOrth}
We first prove a standard result which states that the length of a {\it small} chord is a {\it good} approximation for the length of a curve.
\begin{lem}\label{LOrth}
Let $0<\eta<1$ and let $f\in C^2([0,\eta],\R^+)$. We fix an orthonormal frame and we denote $C_f$ the graph of  $f$ in the orthonormal frame. Let  $A=(a,f(a)),B=(b,f(b))\in C_f$ (with  $0\leq a<b\leq\eta$) and let $\C=[AB]$ be the chord of  $C_f$ joining $A$ and $B$. We denote $\wideparen{AB}$ the arc of $C_f$ with endpoints $A$ and $B$.

We have
\[
\H^1(\C)\leq\H^1(\wideparen{AB})\leq\H^1(\C)\left\{1+(b-a)\|f''\|_{L^\infty}[2\|f'\|_{L^\infty}+\|f''\|_{L^\infty}(b-a)]\right\}.
\]
%uniformément en $A$ et $B$.

%Ici $\ell(\Gamma)$ est la longueur d'une courbe $\Gamma$.
\end{lem}
\begin{proof}The estimate  $\H^1(\C)\leq\H^1(\wideparen{AB})$ is standard, we thus prove the second inequality.

On the one hand 
\[
\H^1(\C)=\sqrt{(a-b)^2+[f(a)-f(b)]^2}=(b-a)\sqrt{1+\left(\dfrac{f(a)-f(b)}{a-b}\right)^2}.
\]
On the other hand
\[
\H^1(\wideparen{AB})=\int_{a}^{b}\sqrt{1+f'^2}.
\]
With the help of the Mean Value Theorem, there exists $c\in(a,b)$ s.t.
\[
\dfrac{f(a)-f(b)}{a-b}=f'(c).
\]
Applying once again the Mean Value Theorem [to  $f'$], for $x\in[a,b]$ there exists $c_x$ between $c$ and $x$ s.t.
\[
f'(x)=f'(c)+f''(c_x)(x-c).
\]

Consequently for $x\in[a,b]$ we have:
\begin{eqnarray*}
\sqrt{1+f'(x)^2}&=&\sqrt{1+[f'(c)+f''(c_x)(x-c)]^2}
\\&=&\sqrt{1+f'(c)^2}\sqrt{1+\dfrac{2f'(c)f''(c_x)(x-c)+f''(c_x)^2(x-c)^2}{1+f'(c)^2}}
\\&\leq&\sqrt{1+\left(\dfrac{f(a)-f(b)}{a-b}\right)^2}\left[1+2\|f'\|_{L^\infty}\|f''\|_{L^\infty}(b-a)+\|f''\|^2_{L^\infty}(b-a)^2\right].
\end{eqnarray*}
Thus we have
\begin{eqnarray*}
\H^1(\wideparen{AB})&=&\int_{a}^{b}\sqrt{1+f'(x)^2}{\,\rm d}x
\\&\leq&(b-a)\sqrt{1+\left(\dfrac{f(a)-f(b)}{a-b}\right)^2}\left[1+2\|f'\|_{L^\infty}\|f''\|_{L^\infty}(b-a)+\|f''\|^2_{L^\infty}(b-a)^2\right]
\\&=&\H^1(\C)\left\{1+(b-a)\|f''\|_{L^\infty}[2\|f'\|_{L^\infty}+\|f''\|_{L^\infty}(b-a)]\right\}.
\end{eqnarray*}
\end{proof}

We now state another technical  lemma which gives an upper bound for the height of the curve w.r.t. its chord. 
\begin{lem}\label{LCorde}
Let $0\leq a< b\leq\eta$, $f\in C^2([0,\eta],\R^+)$ be a strictly concave function and let $C_f$ be the graph of $f$ in an orthonormal frame. Let $A=(a,f(a))$ and $B=( b,f( b))$ be two points of $C_f$. 

Assume that we have $\eta<\dfrac{1}{2\|f'\|_{L^\infty([0,\eta])}\|f''\|_{L^\infty([0,\eta])}}$ in order to define for $C\in[AB]$  [with the help of Lemma \ref{LemUniqnessWelde}]  $\tilde{C}$ as the unique intersection point of $C_f$ with the line orthogonal to $[AB]$ passing by $C$. %tel que $f(0)=f(\delta)=0$ alors

We have
\[
\H^1([C\tilde C])\leq\dfrac{( b-a)^2\|f''\|_{L^\infty}}{8}.
\]
\end{lem}
\begin{proof}
Let $0\leq a< b\leq\eta$, $f\in C^2([0,\eta],\R^+)$ be as in Lemma \ref{LCorde}.

We consider the function
\[
\begin{array}{cccc}
g:&[0,\eta]&\to&\R\\&x&\mapsto& f(x)-\left[\dfrac{f(b)-f(a)}{b-a}(x-a)+f(a)\right]
\end{array}.
\]
It is clear that $g$ is non negative since $f$ is strictly concave.

For $C\in[AB]$, we let $\tilde{C}$ be as in  Lemma \ref{LCorde}. Then we have
\[
\sup_{C\in[AB]}\H^1([C\tilde C])=\max_{[0,\eta]}g.
\]
Thus, it suffices to prove $\max_{[0,\eta]}g\leq\dfrac{( b-a)^2\|f''\|_{L^\infty}}{8}$.

Since $g$ is $C^1$ and $g(a)=g(b)=0$, there exists $c\in(a,b)$ s.t. 
\[
\text{$g(c)=\max_{[0,\eta]}g$ and $g'(c)=0$.}
\] Let $t\in\{a,b\}$ be s.t. $|t-c|\leq\dfrac{b-a}{2}$. Using a Taylor expansion, there exists $\tilde{c}$ between $c$ and $t$ s.t.
\[
0=g(t)=g(c)+(t-c)g'(c)+\dfrac{(t-c)^2}{2}g''(\tilde{c}).
\]
Thus 
\[
0\leq\max_{[0,\eta]}g=g(c)=-\dfrac{(t-c)^2}{2}g''(\tilde{c})\leq\dfrac{( b-a)^2\|f''\|_{L^\infty}}{8}.
\]
The last inequality completes the proof.
\begin{comment}
is nonnegative since $f$ is strictly concave. Moreover $g''=f''$ and thus $g$ is also strictly concave.

we may assume
We first note that if we let  $\tilde x$ be the  abscissa of $\tilde{C}$ and  if we let $P$ be the point of $[AB]$ having $\tilde x$ for  abscissa, then  $\H^1([P\tilde{C}])\geq  \H^1([C\tilde C])$ since the triangle  $CP\tilde{C}$ is right-angled in $C$.

From the Mean Value Theorem, there exist
\begin{itemize}
\item $c_x\in(a,\tilde x)$ s.t. $f'(c_x)=\dfrac{f(\tilde{x})-f(a)}{\tilde{x}-a}$
\item $c\in]a,b[$ s.t. $f'(c)=\dfrac{f( b)-f(a)}{ b-a}$
\item $\tilde{c}\text{ between $c$ and $c_x$}$ s.t. $f''(\tilde c)=\dfrac{f'(c_x)-f'(c)}{ c_x-c}$
\end{itemize}

Moreover
\begin{eqnarray*}
\H^1([P\tilde{C}])&=&\left|f(\tilde{x})-\left[(\tilde{x}-a)\dfrac{f( b)-f(a)}{ b-a}+f(a)\right]\right|\\
&=&|\tilde{x}-a|\left|\dfrac{f(\tilde{x})-f(a)}{\tilde{x}-a}-\dfrac{f( b)-f(a)}{ b-a}\right|
\\
%\left[\begin{array}{c}\text{Mean Value Theorem}\\a<c_x<\tilde{x}\\a<c< b\end{array}\right]
&=&|\tilde{x}-a|\left|f'(c_x)-f'(c)\right|
\\
%\left[\begin{array}{c}\text{Mean Value Theorem}\\\tilde{c}\text{ between $c$ and $c_x$}\end{array}\right]
&=&|\tilde{x}-a||c_x-c|\left|f''(\tilde c)\right|
\\&\leq&( b-a)^2\|f''\|_{L^\infty}
\end{eqnarray*}
which gives the result.
\end{comment}
\end{proof}
\subsection{Proof of Proposition \ref{defKinffty}}\label{ProofOfPropDefKinkdfg}

We prove that 
\begin{equation}\label{LiminfEst}
\liminf_{N\to\infty}\H^1(\K_N)>0.
\end{equation}

{\bf Step 1.} We prove that $\displaystyle\max_{k=1,...,2^N}\H^1(\C_k^N)\leq\left(\dfrac{2}{3}\right)^N$

For $N\geq1$ we let $\{K_k^N\,;\,k=1,...,2^N\}$ be the set of the connected components of $\K_N$. We let $\C_k^N$ be the chord of $K_k^N$ and we define $\mu_N=\max_{k=1,...,2^N}\H^1(\C_k^N)$. Note that by \eqref{MainGeomHyp} we have $\mu_0<1$.

We first prove that for  $N\geq0$ we have
\begin{equation}\label{AuxGoodEst}
\mu_{N+1}\leq\frac{2}{3}\mu_N.
\end{equation}
By induction \eqref{AuxGoodEst} implies  [since to $\mu_0<1$]
\begin{equation}\label{AuxGoodEstReformulate}
\mu_{N}\leq\left(\frac{2}{3}\right)^N.
\end{equation} 
In order to get \eqref{AuxGoodEst}, we prove that for $N\geq1$ and $K_k^N$ a connected component of $\K_N$ and $\C_k^N$ its chord, we have 
\begin{equation}\label{AuxGoodEstAltern}
\H^1(\C)\leq\dfrac{2\H^1(\C_k^N)}{3}\text{ for $\C\in\mathcal{S}(\C_k^N)$}
\end{equation} [see Notation \ref{Not.FatherChord} for $\mathcal{S}(\cdot)$, the set of sons of a chord].

Let $N\geq1$. For $k\in\{1,...,2^N\}$, we let $K_k^N$ be a connected component of $\K_N$.  We let $K_{2k-1}^{N+1},K_{2k}^{N+1}\in\mathcal{S}(K_k^N)$ be the curve obtained from $K_k^N$ in the induction step.

For $\tilde k\in\{2k-1,2k\}$, we let  $\C_{\tilde k}^{N+1}$ be the chords of  $K_{\tilde k}^{N+1}$.

In the frame $\repereO$, we may define four points of $\Gamma$, $(a_1,f(a_1)),(b_1,f(b_1)),(a_2,f(a_2)),(b_2,f(b_2))$, with $0<a_1<b_1<a_2<b_2<\eta$  s.t.:
\begin{itemize}
\item  the endpoints of  $K_{2k-1}^{N+1}$ are $(a_1,f(a_1))\&(b_1,f(b_1))$;
\item the endpoints of $K_{2k}^{N+1}$ are $(a_2,f(a_2))\&(b_2,f(b_2))$;
\item the endpoints of $K_{k}^{N}$ are $(a_1,f(a_1))\&(b_2,f(b_2))$.
\end{itemize}
In the frame $\repereO$ we let also $(\alpha_1,\beta_1),(\alpha_2,\beta_2)$ be the coordinates of the points of $\C_k^N$ s.t. for $l\in\{1,2\}$, the triangles whose vertices are $\{(a_l,f(a_l));(b_l,f(b_l));(\alpha_l,\beta_l)\}$ are right angled in $(\alpha_l,\beta_l)$.

We denote \begin{itemize}
\item $\I_1$ the segment $[(b_1,f(b_1));(\alpha_1,\beta_1)]$; 
\item $\I_2$  the segment $[(a_2,f(a_2));(\alpha_2,\beta_2)]$.
\end{itemize}
From the construction of $K_{2k-1}^{N+1}\&K_{2k}^{N+1}$ and from Pythagorean theorem we have for $l=1,2$
\[
\H^1(\C_{2k-2+l}^{N+1})^2=\H^1(\I_l)^2+\left(\frac{\H^1(\C_k^N)-\H^1(\C_k^N)^2}{2}\right)^2.
\]
Using Lemma \ref{LCorde} we get that 
\[
\H^1(\I_l)\leq ( b_2-a_1)^2\|f''\|_{L^\infty}.
\]
On the other hand we have obviously $b_2-a_1\leq\H^1(\C_k^N)$. Consequently we get
\begin{eqnarray*}
\H^1(\C_{2k-2+l}^{N+1})^2&\leq&\H^1(\C_k^N)^4\|f''\|^2_{L^\infty}+\left(\frac{\H^1(\C_k^N)-\H^1(\C_k^N)^2}{2}\right)^2
\\
&\leq&\H^1(\C_k^N)^4\|f''\|^2_{L^\infty}+\frac{\H^1(\C_k^N)^2}{4}.
\end{eqnarray*}
Therefore
\[
\H^1(\C_{2k-2+l}^{N+1})\leq\dfrac{\H^1(\C_k^N)}{2}\sqrt{1+4\|f''\|^2_{L^\infty}\H^1(\C_k^N)^2},
\]
thus using \eqref{AltFondEq} we get
\[
\H^1(\C_{2k-2+l}^{N+1})\leq\dfrac{2\H^1(\C_k^N)}{3}.
\]
The last estimate gives \eqref{AuxGoodEstAltern} and thus \eqref{AuxGoodEstReformulate} holds.\\

{\bf Step 2.} We prove that $\displaystyle\liminf_{N\to\infty}\sum_{k=1}^{2^N}\H^1(\C^N_k)>0$

For $N\geq1$, we let 
\[
c_N=\sum_{k=1}^{2^N}\H^1(\C^N_k).
\]

The main ingredient in this step consists in noting that, a son of $\C^N_k$ is an hypothenuse of a right angled triangle which admits a  cathetus  of length $\dfrac{\H^1(\C^N_k)-\H^1(\C^N_k)^2}{2}$. 

Consequently we have
\[
\H^1(\C^{N+1}_{2k-1})+\H^1(\C^{N+1}_{2k})\geq \H^1(\C^N_k)-\H^1(\C^N_k)^2.
\]
Thus, summing the previous inequality for $k=1,...,2^N$ we get
\begin{eqnarray*}
c_{N+1}&=&\sum_{k=1}^{2^N}\H^1(\C^{N+1}_{2k-1})+\H^1(\C^{N+1}_{2k})
\geq\sum_{k=1}^{2^N} \H^1(\C^N_k)[1-\H^1(\C^N_k)]
\geq c_N(1-\mu_N)
\geq c_N\left[1-\left(\frac{2}{3}\right)^N\right].
\end{eqnarray*}
By induction for $N\geq2$
\begin{eqnarray*}
c_N\geq c_1\prod_{k=1}^{N-1}\left[1-\left(\frac{2}{3}\right)^k\right]
=c_1\times{\rm exp}\left[\sum_{k=1}^{N-1}\ln\left[1-\left(\frac{2}{3}\right)^k\right]\right].
\end{eqnarray*}
It is clear that $\liminf_N\sum_{l=1}^{N-1}\ln\left[1-\left(\frac{2}{3}\right)^k\right]>-\infty$, thus $\liminf_N c_N>0$.\\

{\bf Step 3.} We prove \eqref{LiminfEst}.

Since for $K^N_k$, a connected component  of $\K_N$, and $\C_k^N$ its chord, we have $\H^1(K_k^N)\geq\H^1(\C_k^N)$, from Step 2 we get \eqref{LiminfEst}.
\section{A fundamental ingredient in the construction of the $\tilde{\Psi}_N$'s}\label{AppeSepTr}
In this section we use the notation of Sections \ref{SectConstructCantor} and \ref{SectionConstructFunction}.
 \begin{lem}\label{PropMainGeomHyp}
Let $\gamma\subset\,\curve{AB}$ be a curve and let $\C$ be its chord. We let  $\gamma_1,\gamma_2$ be the curves included in  $\gamma$ obtained by the induction construction represented Figure \ref{Heredite} [section \ref{SectHerdityCOnstruKantor}]. For $l=1,2$, we denote also by $\C_l$ the chord of $\gamma_l$ and by $T_l$ the right-angled triangle having $\C_l$ as side of the right-angle and having its hypothenuse included in $\C$.

If $\H^1(\C)<\min\{2^{-1},(4\|f''\|_{L^\infty}^2)^{-2}\}$, then the hypothenuses of the triangles $T_1$ and $T_2$ have their length strictly lower than $\dfrac{\H^1(\C)}{2}$. And in particular the triangles $T_1$ and $T_2$ are disjoint.
\end{lem}
\begin{remark}\label{RemPropMainGeomHyp}
From \eqref{MainGeomHyp}, we know that $\C_0=\C_1^0$ is s.t. $\H^1(\C_1^0)<\min\{2^{-1},(4\|f''\|_{L^\infty}^2)^{-2}\}$. From \eqref{AuxGoodEst} we have that for $N\geq1$ and $k\in\{1,...,2^N\}$ we have $\H^1(\C_k^N)<\H^1(\C_1^0)<\min\{2^{-1},(4\|f''\|_{L^\infty}^2)^{-2}\}$. 

Therefore with the help of Lemma \ref{PropMainGeomHyp}, for $N\geq1$, the triangles $T_k^N$'s are pairwise disjoint. 
\end{remark}
\begin{proof}
We model the statement by denoting  $\{M,Q\}$ the set of endpoints of $\gamma$ and $N$ and $P$ are points s.t.:
\begin{itemize}
\item[$\bullet$]$M,N$ are the endpoints of $\gamma_1$ 
\item[$\bullet$] $P,Q$ are the endpoints of $\gamma_2$. 
\end{itemize}
We denote $\delta:=\H^1([MQ])=\H^1(\C)<\min\{2^{-1},(4\|f''\|_{L^\infty}^2)^{-2}\}$.

We fix an orthonormal frame $\tilde\repere$ with the origin in $M$, with the $x$-axis  $(MQ)$ and s.t. $N,P,Q$ have respectively for coordinates  $(x_1,y_1)$, $(x_2,y_2)$ and $(x_3,0)$ where $0<x_1<x_2<x_3$ and $y_1,y_2>0$. 

By construction we have
\[
x_1=\dfrac{\delta-\delta^2}{2},\,x_2=\dfrac{\delta+\delta^2}{2}\text{ and }x_3=\delta.
\]

Moreover, arguing as in the proof of Lemma \ref{LCorde} we have [recall that $\curve{AB}$ is the graph of a function $f$ in an other orthonormal frame]:
\[
0< y_1,y_2\leq\delta^2\|f''\|_{L^\infty}.
\]
\begin{figure}[h]
\psset{xunit=1.0cm,yunit=1.0cm,algebraic=true,dimen=middle,dotstyle=o,dotsize=3pt 0,linewidth=0.8pt,arrowsize=3pt 2,arrowinset=0.25}
\begin{pspicture*}(-1,-1)(15,3)
\psline(0,0)(6,2)
\psline(7.6,1.54)(14,0)
\psline(6,2)(6.67,0)
\psline(7.6,1.54)(7.23,0)
\psline(0,0)(14,0)
\begin{scriptsize}
\psdots[dotstyle=*,linecolor=darkgray](0,0)
\rput[bl](-.5,0.12){{$(0,0)$}}
\psdots[dotstyle=*](6,2)
\rput[bl](5.5,2.12){{$(x_1,y_1=ax_1)$}}
\psdots[dotstyle=*](7.6,1.54)
\rput[bl](7.2,1.66){{$(x_2,y_2=\alpha x_2+\beta)$}}
\psdots[dotstyle=*](14,0)
\rput[bl](14.08,0.12){{$(x_3,0)$}}
\rput[bl](3.1,1.5){$y=ax$}
\rput[bl](10.72,1){$y=\alpha x+\beta$}
\psdots[dotstyle=*,linecolor=darkgray](6.67,0)
\rput[bl](6.2,-0.5){{$(x_4,0)$}}
\psdots[dotstyle=*,linecolor=darkgray](7.23,0)
\rput[bl](7.3,-0.5){{$(x_5,0)$}}
\psline(6,2)(5.66,1.89)
\psline(5.66,1.89)(5.77,1.54)
\psline(5.77,1.54)(6.11,1.66)
\psline(6.11,1.66)(6,2)
\psline(7.6,1.54)(7.53,1.25)
\psline(7.53,1.25)(7.82,1.18)
\psline(7.82,1.18)(7.89,1.47)
\psline(7.89,1.47)(7.6,1.54)
%\rput[bl](6.02,0.9){$e$}
%\rput[bl](7.1,0.86){$f$}
%\rput[bl](7,-0.32){$g$}
\end{scriptsize}
\end{pspicture*}\caption{Model problem}
\end{figure}
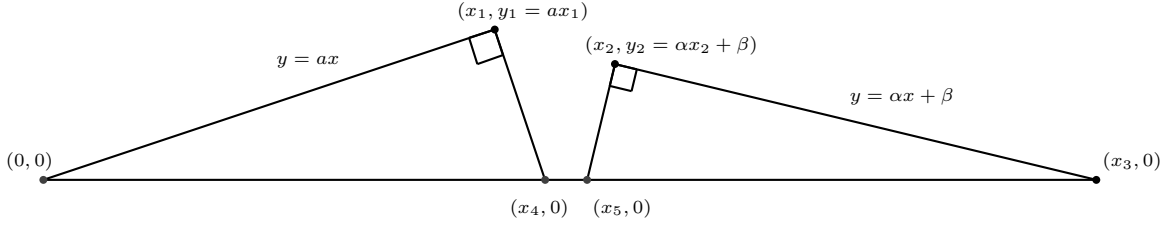

From these points, in Section \ref{SectHerdityCOnstruKantor}, we defined two right-angled triangles having their hypothenuses contained in the $x$-axis.
 
 The first triangle admits for vertices  the origin $(0,0)$, $(x_1,y_1)$ and a point of the $x$-axis $(x_4,0)$. This triangle is right angled in $(x_1,y_1)$. In the frame $\tilde\repere$, one of the side of the right-angle is included in the line parametrized by the cartesian equation  $y=ax$. Since $\delta\leq1/2$
 \[
 |a|=\left|\dfrac{y_1}{x_1}\right|\leq\dfrac{2\delta^2\|f''\|_{L^\infty}}{\delta-\delta^2}\leq4\|f''\|_{L^\infty}\delta.
 \]
  
The second triangle admits for vertices $(x_2,y_2)$, $(x_3,0)$ and a point of the $x$-axis $(x_5,0)$. This triangle is right-angled in $(x_2,y_2)$. In the frame $\tilde\repere$, one of the side of the right-angle is included in the line parametrized by the cartesian equation $y=\alpha x+\beta$ where 
 \[
 |\alpha|=\left|\dfrac{y_2}{x_2-x_3}\right|\leq\dfrac{2\delta^2\|f''\|_{L^\infty}}{\delta-\delta^2}\leq4\|f''\|_{L^\infty}\delta.
 \]%(dans la suite de ce travail $\alpha\sim0$).
 
The  proof of the proposition consists in obtaining
\[
x_4<\dfrac{x_3}{2}\text{ and }x_3-x_5< \dfrac{x_3}{2}.
\]

We get the first estimate. With the help of Pythagorean theorem we have
\[
x_1^2+y_1^2+(x_1-x_4)^2+y_1^2=x_4^2.
\]
By noting that $y_1=ax_1$ we have 
\[
x_4=(1+a^2)x_1.
\]
Thus:
\begin{eqnarray*}
x_4<\dfrac{x_3}{2}&\Longleftrightarrow&(1+a^2)\dfrac{\delta-\delta^2}{2}<\dfrac{\delta}{2}
\\
&\Longleftarrow&(1+16\|f''\|_{L^\infty}^2\delta^2)(1-\delta)<1
%\\
%&\Longleftrightarrow&(1+16\|f''\|_{L^\infty}^2\delta^2)x_1<\dfrac{\delta}{2}
\\
&\Longleftrightarrow&\delta-\delta^2<\dfrac{1}{16\|f''\|_{L^\infty}^2}
\\
&\Longleftarrow&\delta<\dfrac{1}{16\|f''\|_{L^\infty}^2}.
\end{eqnarray*}

Following the same strategy we get that if  $\delta<\dfrac{1}{16\|f''\|_{L^\infty}^2}$ then $x_3-x_5< \dfrac{x_3}{2}$.
\begin{comment}

De même on a 
\[
(x_2-x_3)^2+y_2^2+(x_2-x_5)^2+y_2=(x_3-x_5)^2
\]
et donc en notant que $y_2=\alpha x_2+\beta$ avec $\beta=-\alpha x_3$ on obtient:
\[
x_5=(1+\alpha^2)x_2-\alpha^2x_3.
\]
Rappelons que notre but est de montrer que  $x_4<x_5$ ce qui revient à 
\begin{eqnarray*}
x_4<x_5&\Leftrightarrow&(1+a^2)x_1<(1+\alpha^2)x_2-\alpha^2x_3
\\
&\Leftrightarrow&x_2-x_1>\alpha^2(x_3-x_2)+a^2x_1.
\end{eqnarray*}
D'une part on a $x_2-x_1=\delta^2$. Et d'autre part:
\[
\alpha^2(x_3-x_2)+a^2x_1\leq4\|f''\|^2_{L^\infty}\delta^2(\delta-\delta^2).
\]
Ainsi, afin d'obtenir que $x_4<x_5$, il suffit d'avoir
\[
4\|f''\|_{L^\infty}^2\delta^2(\delta-\delta^2)\leq\delta^2
\]
ce qui est clairement vérifié si $\delta<(4\|f''\|_{L^\infty}^2)^{-1}$.
\end{comment}
\end{proof}
\section{Adaptation of a result of Giusti in \cite{giusti1984minimal}}\label{AppendixGiusti}
In this appendix we present briefly the proof of Theorem 2.16 and Remark 2.17 in \cite{giusti1984minimal}. The argument we present below follows the proof of Theorem 2.15 in \cite{giusti1984minimal}.
\begin{prop}\label{ExtGiustiUpBound}
Let $\O\subset\R^n$ be a bounded open set of class $C^2$ and let $h\in L^1(\p\O)$. For all $\v>0$ there exists $u_\v\in W^{1,1}(\O)$ s.t. $\tr_{\p\O}u_\v=h$ and
\[
\|u_\v\|_{W^{1,1}(\O)}:=\|u_\v\|_{L^1(\O)}+\|\n u_\v\|_{L^1(\O)}\leq(1+\v)\|h\|_{L^1(\O)}.
\] 
\end{prop}
\begin{proof}
We sketch the proof of Proposition \ref{ExtGiustiUpBound}. Let $h\in L^1(\p\O)$ and let $\v>0$ be sufficiently small s.t.
\[
(1+\v^2)^2+\v^2+\v^4<1+\frac{\v}{2}\text{ and }(1+\v^2)\v^2<\frac{\v}{2}.
\].

{\bf Step 1.} We may consider $\eta>0$ sufficiently small s.t. in $\O_\eta:=\{x\in\O\,;\,\dist(x,\p\O)<\eta\}$ we have:
\begin{enumerate}
\item The function 
\[
\begin{array}{cccc}
d:&\O_\eta&\to&(0,\eta)\\&x&\mapsto&\dist(x,\p\O)
\end{array}
\]
is of class $C^1$ and satisfies $|\n d|\geq1/2$,
\item The orthogonal projection on $\p\O$, $\Pi_{\p\O}$, is Lipschitz.
\end{enumerate}

We now fix a sequence $(h_k)_k\subset C^\infty(\p\O)$ s.t. $h_k\stackrel{L^1}{\to}h$. We may assume that (up to replace the first term and to consider an extraction):
\begin{enumerate}
\item $h_0\equiv0$,
\item $\sum_{k\geq0}\|h_{k+1}-h_k\|_{L^1}\leq(1+\v^2)\|h\|_{L^1}$.
\end{enumerate}
And finally we fix a decreasing sequence $(t_k)_k\subset\R^*_+$ s.t.
\begin{enumerate}
%\item $\lim_{k\to\infty}t_k=0$
\item $t_0<\min(\eta,\v^2)$ is sufficiently small s.t. 
\begin{itemize}
\item $4t_0\max(1;\|\n\Pi_{\p\O}\|_{L^\infty})\times\max(1,\sup_k\|h_k\|_{L^1})<\min(\v^2,\v^2\|h\|_{L^1})$,
\item for $\varphi\in L^1(\p\O)$ we have for $s\in(0,t_0)$
\[
\int_{d^{-1}(\{s\})}|\varphi\circ\Pi_{\p\O}(x)|\leq(1+\v^2)\int_{\p\O}|\varphi(x)|.
\]
\end{itemize}
\item For $k\geq1$ we have $t_k\leq\dfrac{t_0\|h\|_{L^1}}{2^k(1+\|\n h_k\|_{L^\infty}+\|\n h_{k+1}\|_{L^\infty})}$.
\end{enumerate} 
{\bf Step 2.} We define
\[\begin{array}{cccc}
u_\v:&\O&\to&\R
\\
&x&\mapsto&\begin{cases}\dfrac{d(x)-t_{k+1}}{t_k-t_{k+1}}h_k\circ\Pi_{\p\O}(x)+\dfrac{t_{k}-d(x)}{t_k-t_{k+1}}h_{k+1}\circ\Pi_{\p\O}(x)&\text{if }d(x)\in[t_{k+1},t_k)\\0&\text{otherwise} \end{cases}\end{array}.
\]
We may easily check that $u_\v$ is locally Lipschitz and thus weakly differentiable. 

From the coarea formula and a standard change of variable we have
\begin{eqnarray*}
\|u_\v\|_{L^1}&\leq&2\int_{\{d\leq t_0\}}|u_\v||\n d|
\\&\leq&2\int_0^{t_0}{\rm d}s\int_{d^{-1}(\{s\})}|u_\v|{\rm d}x
\\&\leq&2\sum_{k\geq0}\int_{t_{k+1}}^{t_k}{\rm d}s\int_{d^{-1}(\{s\})}|u_\v|{\rm d}x
\\&\leq&2\sum_{k\geq0}\int_{t_{k+1}}^{t_k}{\rm d}s\int_{d^{-1}(\{s\})}[|h_{k}\circ\Pi_{\p\O}(x)|+|h_{k+1}\circ\Pi_{\p\O}(x)|]{\rm d}x
\\&\leq&2(1+\v^2)\sum_{k\geq0}\int_{t_{k+1}}^{t_k}{\rm d}s\int_{\p\O}[|h_{k}(x)|+|h_{k+1}(x)|]{\rm d}x
\\&\leq&2(1+\v^2)\sum_{k\geq0}(t_k-t_{k+1})(\| h_k\|_{L^1}+\| h_{k+1}\|_{L^1})
\\&\leq&4(1+\v^2)t_0\sup_k\|h_k\|_{L^1}
\\&\leq&(1+\v^2)\v^2\|h\|_{L^1}
\\&\leq&\frac{\v}{2}\|h\|_{L^1}.
\end{eqnarray*}

We now estimate $\|\n u_\v\|_{L^1}$. It is easy to check that if $d(x)\in(t_{k+1},t_k)$ then we have
\[
|\n u_\v(x)|\leq|\n d(x)|\left[\dfrac{| h_k\circ\Pi_{\p\O}(x)- h_{k+1}\circ\Pi_{\p\O}(x)|}{t_k-t_{k+1}}+2\|\n\Pi_{\p\O}\|_{L^\infty}\left[|\n h_k|\circ\Pi_{\p\O}(x)+|\n h_{k+1}|\circ\Pi_{\p\O}(x)\right]\right].
\]
Consequently we get
\begin{eqnarray*}
\|\n u_\v\|_{L^1}&\leq&(1+\v^2)\sum_{k\geq0}\left\{\int_{t_{k+1}}^{t_k}\dfrac{\|h_{k+1}-h_k\|_{L^1}}{t_k-t_{k+1}}+2\|\n\Pi_{\p\O}\|_{L^\infty}(t_k-t_{k+1})(\|\n h_{k+1}\|_{L^1}+\| \n h_k\|_{L^1})\right\}
\\&\leq&(1+\v^2)[(1+\v^2)\|h\|_{L^1}+2\|\n\Pi_{\p\O}\|_{L^\infty}t_0\|h\|_{L^1}]
\\&\leq&(1+\v^2)[(1+\v^2)+\v^2]\|h\|_{L^1}
%\\&\leq&(1+\v^2)[(1+\v^2)+\v^2]\|h\|_{L^1}
\\&\leq&(1+\v/2)\|h\|_{L^1}.
\end{eqnarray*}
Consequently $u_\v\in W^{1,1}(\O)$ and $\|u_\v\|_{W^{1,1}}\leq(1+\v)\|h\|_{L^1}$.

In order to end the proof it suffices to check that $\tr_{\p\O}(u_\v)=h$. The justification of this property follows the argument of Lemma 2.4 in \cite{giusti1984minimal}. 
\end{proof}

\noindent{\bf Acknowledgements.} %\normalsize
The author would like to thank Petru Mironescu for fruitful discussions.

\bibliographystyle{amsalpha}
\bibliography{biblioNonExistence}

\end{document}